\documentclass[11pt]{amsart}
\allowdisplaybreaks[4]
\usepackage{amssymb}
\usepackage{anysize}
\usepackage{hyperref}
\usepackage{extarrows}

\numberwithin{equation}{section}
\newtheorem{theorem}{Theorem}[section]
\newtheorem{lemma}[theorem]{Lemma}
\newtheorem{prop}{{\bf Proposition}}[section]

\theoremstyle{definition}

\theoremstyle{remark}
\newtheorem{remark}{{\bf Remark}}[section]
\newcommand{\be}{\begin{equation}}
\newcommand{\ee}{\end{equation}}
\newcommand{\bea}{\begin{eqnarray*}}
\newcommand{\eea}{\end{eqnarray*}}

\makeatletter

\newcommand{\Rmnum}[1]{\expandafter\@slowromancap\romannumeral #1@}
\makeatother
\makeatletter 
\@addtoreset{equation}{section}
\makeatother  

\marginsize{35mm}{35mm}{38mm}{45mm}

\begin{document}

\title[]{Tracy-Widom law for the extreme eigenvalues of sample correlation matrices}
\author{Zhigang Bao}
\author{Guangming Pan}
\author{Wang Zhou}
\
\thanks{ Z.G. Bao was partially supported by NSFC grant
11071213, NSFC grant 11101362, ZJNSF grant R6090034 and SRFDP grant 20100101110001;
G.M. Pan was partially supported by the Ministry of Education, Singapore, under grant \# ARC 14/11;
   W. Zhou was partially supported by the Ministry of Education, Singapore, under grant \# ARC 14/11,  and by a grant
R-155-000-116-112 at the National University of Singapore.}

\address{Department of Mathematics, Zhejiang University, P. R. China}
\email{zhigangbao@zju.edu.cn}

\address{Division of Mathematical Sciences, School of Physical and Mathematical Sciences, Nanyang Technological University, Singapore 637371}
\email{gmpan@ntu.edu.sg}

\address{Department of Statistics and Applied Probability, National University of
 Singapore, Singapore 117546}
\email{stazw@nus.edu.sg}

\subjclass[2010]{15B52, 62H25, 62H10}

\date{}

\keywords{extreme eigenvalues, sample correlation matrices, sample covariance matrices, Stieltjes transform, Tracy-Widom law}

\maketitle

\begin{abstract}
Let the sample correlation matrix be $W=YY^T$, where $Y=(y_{ij})_{p,n}$ with
$y_{ij}=x_{ij}/\sqrt{\sum_{j=1}^nx_{ij}^2}$. We assume $\{x_{ij}: 1\leq i\leq p, 1\leq j\leq n\}$ to be a collection of independent symmetric distributed random variables with sub-exponential tails. Moreover, for any $i$, we assume $x_{ij}, 1\leq j\leq n$ to be identically distributed. We assume $0<p<n$ and $p/n\rightarrow y$ with some $y\in(0,1)$ as $p,n\rightarrow\infty$. In this paper, we provide the Tracy-Widom law ($TW_1$)
for both the largest and smallest eigenvalues of $W$. If $x_{ij}$ are i.i.d. standard normal, we can derive the $TW_1$ for both the largest and smallest eigenvalues of the matrix
$\mathcal{R}=RR^T$, where $R=(r_{ij})_{p,n}$ with $r_{ij}=(x_{ij}-\bar x_i)/\sqrt{\sum_{j=1}^n(x_{ij}-\bar x_i)^2}$, $\bar x_i=n^{-1}\sum_{j=1}^nx_{ij}$.
\end{abstract}
\maketitle
\section{Introduction}
Suppose we have a $p$-dimensional distribution with mean $\mu$ and covariance
matrix $\Sigma$. In recent three or four decades, in many research areas,
including signal processing, network security, image processing, genetics, stock
marketing and other economic problems, people are interested in the case
where $p$ is quite large or proportional to the sample size. Naturally, one may ask
how to test the independence among the $p$ components of the population. From the principal component analysis point of view, the independence test statistic is usually the maximum eigenvalue  of the sample covariance matrices. Under the
additional normality assumption, Johnstone \cite{John01} derived the asymptotic distribution
of the largest eigenvalue of the sample covariance matrices to study the test
$H_0: \Sigma=I$ assuming $\mu=0$.

However, sample covariance matrices are not scale-invariant. So if $\mu=0$, Johnstone \cite{John01} proposes to perform principal component analysis (PCA) by the maximum eigenvalue of the matrix $W=YY^T$, where
\begin{eqnarray}
Y=(y_{ij})_{p, n}:=\left(
\begin{array}{cccc}
\frac{x_{11}}{||\mathbf{x}_{1}||} &\frac{x_{12}}{||\mathbf{x}_{1}||} &\cdots &\frac{x_{1n}}{||\mathbf{x}_{1}||}\\
\vdots &\vdots &\vdots &\vdots\\
\frac{x_{p1}}{||\mathbf{x}_p||} &\frac{x_{p2}}{||\mathbf{x}_p||} &\cdots &\frac{x_{pn}}{||\mathbf{x}_p||}
\end{array}
\right).
\label{1.1}
\end{eqnarray}
Here $\mathbf{x}_{i}=(x_{i1},\cdots,x_{in})^T$ contains $n$ observations for the $i$-th component of the population, $i=1,\cdots,p$, and $||\cdot||$ represents the vector norm.

Performing PCA on $W$ amounts to PCA on the sample correlations of the original data if $\mu=0$.
So for simplicity, we call $W$ the {\it sample correlation matrix} in this paper.
From now on, the eigenvalues of $W$ will be denoted by
$$0\leq \lambda_1\leq\lambda_2\leq\cdots\leq\lambda_p.$$ Then the empirical distribution (ESD) of $W$ is defined by
\begin{eqnarray*}
F_p(x)=\frac1p\sum_{i=1}^p\mathbf{1}_{\{\lambda_i\leq x\}}.
\end{eqnarray*} The asymptotic property of $F_p$ was studied in \cite{Jiang} and \cite{BZ}. For the almost sure convergence of $\lambda_1$ and $ \lambda_p$, see \cite{Jiang}. \\

 In this paper, we will  study the fluctuations of the extreme eigenvalues $\lambda_1, \lambda_p$ of $W$ for a general population, including multivariate normal one.
The basic assumption on the distribution of our population throughout the paper is\\\\
$\bf{Condition~C_1}$. We assume $x_{ij}$ are independent symmetric distributed random variables with variance $1$. And for any $i$, we assume $x_{i1},\cdots,x_{in}$ to be i.i.d. Furthermore, we request the distributions of the $x_{ij}'$s have sub-exponential tails, i.e., there
exist positive constants $C,C'$ such that for all $1\leq i\leq p, 1\leq j\leq n$ one has
\begin{eqnarray*}
\mathbb{P}(|x_{ij}|\geq t^C)\leq e^{-t}
\end{eqnarray*}
for all $t\geq C'$. And we also assume $p/n\rightarrow y$ as $p,n=n(p)\rightarrow \infty$, where $0<y<1$.

\begin{remark}
 We use $C,C_0,C_1,C_2,C',O(1)$ to denote some positive constants independent of $p$, which may differ from line to line. And we use $C_{\alpha}$ to denote some positive constants depending on the parameter $\alpha$. The notation $||\cdot||_{op},||\cdot||_F$ represent the operator norm and Frobenius norm of a matrix respectively. And $||\cdot||$ represents a Euclidean norm of a vector.
\end{remark}

\begin{remark}
The sample correlation matrix $W$ is invariant under the scaling on the elements $x_{ij}$, so the assumption $Var(x_{ij})=1$ is not necessary indeed. We specify it to be $1$ here just for convenience. Owing to the exponential tails, we can always truncate the variables so that $|x_{ij}|\leq K$ with some $K\geq \log^{O(1)}n$.
\end{remark}

A special sample correlation matrix model is the Bernoulli case, i.e. $x_{ij}$ takes value of $1$ or $-1$ with equal probability. Notice that if $x_{ij}$ are Bernoulli, we always have for all $1\leq i\leq p$
\begin{eqnarray*}
||\mathbf{x}_i||^2=x_{i1}^2+\cdots+x_{in}^2=n.
\end{eqnarray*}
As a consequence, the sample correlation matrix with Bernoulli elements coincides with its corresponding sample covariance matrix for which the limiting distribution of the extreme eigenvalues are well known under some moment assumptions. One can refer to \cite{BF},\cite{FS}, \cite{Peche}, \cite{Sosh} and \cite{WK}. We only summarize their results for the special Bernoulli case as the following theorem.

\begin{theorem}\label{th.0.1}\emph{(Bernoulli case)}
For the matrix $W$ in (\ref{1.1}), if $x_{ij}$ are $\pm 1$ Bernoulli variables, we have
\begin{eqnarray*}
\frac{n\lambda_p-(p^{1/2}+n^{1/2})^2}{(n^{1/2}+p^{1/2})(p^{-1/2}+n^{-1/2})^{1/3}}\stackrel d\longrightarrow TW_1,
\end{eqnarray*}
and
\begin{eqnarray*}
\frac{n\lambda_1-(p^{1/2}-n^{1/2})^2}{(n^{1/2}-p^{1/2})(p^{-1/2}-n^{-1/2})^{1/3}}\stackrel d\longrightarrow TW_1.
\end{eqnarray*}
as $p,n\rightarrow\infty$ with $p/n\rightarrow y\in(0,1)$.
\end{theorem}

Here $TW_1$ is the famous Tracy-Widom distribution of type $1$, which was firstly raised by Tracy and Widom in \cite{TW} for the Gaussian orthogonal ensemble. The distribution function $F_1(t)$ of $TW_1$ admits the representation
\begin{eqnarray*}
F_1(t)=\exp(-\frac12\int_t^\infty [q(x)+(x-t)q(x)^2]dx),
\end{eqnarray*}
where $q$ statisfies the Painlev$\acute{e}$ $II$ equation
\begin{eqnarray*}
q''=tq+2q^3,\quad q(t)\sim \rm{Ai}(t), as~t\rightarrow\infty.
\end{eqnarray*}
Here $\rm{Ai}(t)$ is the Airy function.

The main purpose of this paper is to generalize Theorem \ref{th.0.1} to the population satisfying the basic condition $\mathbf{C}_1$. Our main results are the following two theorems.

\begin{theorem} \label{th.1.2}Let $W$ be a sample correlation matrix satisfying the basic condition $\mathbf{C}_1$. We have
\begin{eqnarray*}
\frac{n\lambda_p-(p^{1/2}+n^{1/2})^2}{(n^{1/2}+p^{1/2})(p^{-1/2}+n^{-1/2})^{1/3}}\stackrel d\longrightarrow TW_1.
\end{eqnarray*}
and
\begin{eqnarray*}
\frac{n\lambda_1-(p^{1/2}-n^{1/2})^2}{(n^{1/2}-p^{1/2})(p^{-1/2}-n^{-1/2})^{1/3}}\stackrel d\longrightarrow TW_1.
\end{eqnarray*}
as $p\rightarrow\infty$.
\end{theorem}

\begin{remark}
For technical reasons, it is convenient to work with the continuous random variables $x_{ij}$. As a result, the events such as eigenvalue collision will only occur with probability zero (see Lemma \ref{le.rrr}). Because  none of our bounds depends on how continuous the $x_{ij}$ are, one can recover
the discrete case from the continuous one by a standard limiting argument by using Weyl's inequality (see Lemma \ref{le.2.2}), especially for the Bernoulli case.
\end{remark}

If the population is normal, then we can derive the Tracy-Widom law for both the largest and smallest eigenvalues of the matrix
$\mathcal{R}=RR^T$, where
\begin{eqnarray}
R=(r_{ij})_{p, n}:=\left(
\begin{array}{cccc}
\frac{x_{11}-\bar x_1}{||\mathbf{x}_{1}-\bar x_1||} &\frac{x_{12}-\bar x_1}{||\mathbf{x}_{1}-\bar x_1||} &\cdots &\frac{x_{1n}-\bar x_1}{||\mathbf{x}_{1}-\bar x_1||}\\
\vdots &\vdots &\vdots &\vdots\\
\frac{x_{p1}-\bar x_p}{||\mathbf{x}_p-\bar x_p||} &\frac{x_{p2}-\bar x_p}{||\mathbf{x}_p-\bar x_p||} &\cdots &\frac{x_{pn}-\bar x_p}{||\mathbf{x}_p-\bar x_p||}
\end{array}
\right).
\label{r}
\end{eqnarray}
Here $\bar x_i=n^{-1}\sum_{j=1}^n x_{ij}$ and $\mathbf{x}_i-\bar x_i$ means each element $x_{ij}$ of $\mathbf{x}_i$ will be subtracted by $\bar x_i$, $i=1,\cdots,p$.
We denote the ordered eigenvalues of $\mathcal{R}$ by $0\leq \lambda_1(\mathcal{R})\leq\cdots\leq\lambda_p(\mathcal{R})$ below.
Actually $\mathcal{R}$ is the sample correlation matrix when the population mean is unknown.
\begin{theorem} \label{th.1.4} For the sample correlation matrix $\mathcal{R}$ with i.i.d $N(0,1)$ elements, if $p/n\rightarrow y\in(0,1)$, we have
\begin{eqnarray*}
\frac{n\lambda_p(\mathcal{R})-(p^{1/2}+n^{1/2})^2}{(n^{1/2}+p^{1/2})(p^{-1/2}+n^{-1/2})^{1/3}}\stackrel d\longrightarrow TW_1.
\end{eqnarray*}
and
\begin{eqnarray*}
\frac{n\lambda_1(\mathcal{R})-(p^{1/2}-n^{1/2})^2}{(n^{1/2}-p^{1/2})(p^{-1/2}-n^{-1/2})^{1/3}}\stackrel d\longrightarrow TW_1
\end{eqnarray*}
as $p\rightarrow\infty$.

\end{theorem}

Throughout the paper, we will use the following ad hoc definitions on the frequent events provided in \cite{TV1}.\\\\
$\bf{Definition~1}$(Frequent events). \cite{TV1} Let $E$ be an event depending on $n$.\\\\
$\bullet$ $E$ holds \emph{asymptotically almost surely} if $\mathbb{P}(E)=1-o(1)$.\\
$\bullet$ $E$ holds with \emph{high probability} if $\mathbb{P}(E)\geq 1-O(n^{-c})$ for some constant $c>0$ (independent of $n$).\\
$\bullet$ $E$ holds with \emph{overwhelming probability} if $\mathbb{P}(E)\geq O_C(n^{-C})$ for every constant $C>0$ (or equivalently, that $\mathbb{P}(E)\geq 1-\exp(-\omega\log n)$).\\
$\bullet$ $E$ holds \emph{almost surely} if $\mathbb{P}(E)=1$.\\

The main strategy is to prove a so-called ``Green function comparison theorem'', which was raised by Erd\"{o}s, Yau and Yin in \cite{EYY} for generalized Wigner matrices. We will provide a ``Green function comparison theorem'' to the sample correlation matrices obeying the assumption $\mathbf{C}_1$ in Section 4, see Theorem \ref{th.4.1}. Then by the comparison theorem, we can compare the general distributed case with the Bernoulli case to get Theorem \ref{th.1.2}. And as an application, we can also get Theorem \ref{th.1.4}.

Our article is organized as follows. In Section 2, we state some basic tools, which can be also found in the series work \cite{TV1}, \cite{TV2},\cite{TV} and \cite{WK}. And we provide some main technical lemmas and theorems in Section 3. The most important one is the so-called delocalization property of singular vectors, which will be shown as an obstacle to establish the Green function comparison theorem in the sample correlation matrices case. And in Section 4, we provide a Green function comparison theorem to prove the edge universality for sample correlation matrices satisfying the assumption $\mathbf{C}_1$. In Section 5, we state the proofs for our main results: Theorem \ref{th.1.2} and Theorem \ref{th.1.4}.\\\\
\section{Basic Tools}
In this section, we state some basic tools from linear algebra and probability theory. Firstly, we denote the ordered singular values of $Y$ by
$$0\leq\sigma_1\leq\sigma_2\leq\cdots\leq\sigma_p,$$
then we have $\sigma_i=\lambda_i^{1/2}$. If we further denote the unit right singular vector of $Y$ corresponding $\sigma_i$ by $u_i$ and the left one by $v_i$, we have
\begin{eqnarray}
Yu_i=\sigma_iv_i \label{1.2}
\end{eqnarray}
and
\begin{eqnarray}
Y^Tv_i=\sigma_iu_i. \label{1.3}
\end{eqnarray}

Below we shall state some tools for eigenvalues, singular values and singular vectors without proof.
\begin{lemma} \emph{(Cauchy's interlacing law)}. Let $1\leq p\leq n$\\

\emph{(i)} If $A_n$ is an $n\times n$ Hermitian matrix, and $A_{n-1}$ is an $n-1\times n-1$ minor, then $\lambda_i(A_n)\leq \lambda_i(A_{n-1})\leq \lambda_{i+1}(A_n)$ for all $1\leq i<n$.

\emph{(ii)} If $A_{n,p}$ is a $p\times n$ matrix, and $A_{n,p-1}$ is a $p-1\times n$ minor, then $\sigma_i(A_{n,p})\leq \sigma_i(A_{n,p-1})\leq \sigma_{i+1}(A_{n,p})$ for all $1\leq i<p$.

\emph{(iii)} If $p<n$, $A_{n,p}$ is a $p\times n$ matrix, and $A_{n-1,p}$ is a $p\times n-1$ minor, then $\sigma_{i-1}(A_{n,p})\leq\sigma_i(A_{n-1,p})\leq\sigma_i(A_{n,p})$ for all $1\leq i\leq p$, with the understanding that $\sigma_0(A_{n,p})=0$. (For $p=n$,
one can consider its transpose and use \emph{(ii)} instead.)
\end{lemma}
\begin{lemma}\label{le.2.2}\emph{(Weyl's inequality)} Let $1\leq p\leq n$\\

$\bullet$ If $M,N$ are $n\times n$ Hermitian matrices, then $||\lambda_i(M)-\lambda_i(N)||\leq ||M-N||_{op}$ for all $1\leq i\leq n$.

$\bullet$ If $M,N$ are $p\times n$ matrices, then $||\sigma_i(M)-\sigma_i(N)||\leq ||M-N||_{op}$  for all $1\leq i\leq p$.
\end{lemma}
The following lemma is on the components of a singular vector, which can be found in \cite{TV}.
\begin{lemma}\label{le.2.0}\cite{TV}
Let $p,n\geq 1$, and let
\begin{eqnarray*}
A_{p,n}=(A_{p,n-1}~~h)
\end{eqnarray*}
be a $p\times n$ matrix with $h\in\mathbb{C}^p$, and let $\binom{u}{x}$ be a right unit singular vector of $A_{p,n}$ with singular value $\sigma_i(A_{p,n})$, where $x\in \mathbb{C}$ and $u\in\mathbb{C}^{n-1}$. Suppose that none of the singular values of $A_{p,n-1}$ is equal to $\sigma_i(A_{p,n})$. Then
\begin{eqnarray*}
|x|^2=\frac{1}{1+\sum_{j=1}^{\min(p,n-1)}\frac{\sigma_j(A_{p,n-1})^2}{(\sigma_j(A_{p,n-1})^2-\sigma_i(A_{p,n})^2)^2}|v_j(A_{p,n-1})\cdot h|^2},
\end{eqnarray*}
 where $\{v_1(A_{p,n-1}),\cdots, v_{\min(p.n-1)}(A_{p,n-1})\in\mathbb{C}^p\}$ is an orthonormal system of left singular vectors corresponding to the non-trivial singular values of $A_{p,n-1}$ and \\ $v_j(A_{p,n-1})\cdot h=v_j(A_{p,n-1})^*h$ with $v_j(A_{p,n-1})^*$ being the complex conjugate of $v_j(A_{p,n-1}).$ \\

 Similarly, if
 \begin{eqnarray*}
 A_{p,n}=\binom{A_{p-1,n}}{l^*}
 \end{eqnarray*}
 for some $l\in \mathbb{C}^n$, and $(v^T,y)^T$ is a left unit singular vector of $A_{p.n}$ with singular value $\sigma_i(A_{p,n})$, where $y\in \mathbb{C}$ and $v\in\mathbb{C}^{p-1}$, and none of the singular values of $A_{p-1,n}$ are equal to $\sigma_i(A_{p,n})$, then
 \begin{eqnarray*}
 |y|^2=\frac{1}{1+\sum_{j=1}^{\min(p-1,n)}\frac{\sigma_{j}(A_{p-1,n})^2}{(\sigma_j(A_{p-1,n})^2-\sigma_i(A_{p,n})^2)^2}|u_j(A_{p-1,n})\cdot l|^2},
 \end{eqnarray*}
 where $\{u_1(A_{p-1,n}),\cdots, u_{\min(p-1,n)}(A_{p-1,n})\in\mathbb{C}^n\}$ is an orthonormal system right singular vectors corresponding to the non-trivial singular values of $A_{p-1,n}$.
\end{lemma}
Further, we need a frequently used tool in the Random Matrix Theory: the Stieltjes transform of ESD $F_{p}(x)$, which is defined by
\begin{eqnarray*}
s_p(z)=\int\frac{1}{x-z}dF_p(x)
\end{eqnarray*}
for any $z=E+i\eta$ with $E\in\mathbb{R}$ and $\eta>0$. If we introduce the Green function $G(z)=(W-z)^{-1}$, we also have
\begin{eqnarray}
s_p(z)=\frac1p TrG(z)=\frac1p\sum_{k=1}^{p}G_{kk}. \label{2.3}
\end{eqnarray}
Here we denote $G_{jk}$ as the $(j,k)$ entry of $G(z)$. As is well known, the convergence of a tight
probability measure sequence is equivalent to the convergence of its Stieltjes transform
sequence towards the corresponding transform of the limiting measure. So corresponding to the convergence of $F_p(x)$ towards $F_{MP,y}(x)$,
the famous Mar\u{c}enko-Pastur law $F_{MP,y}(x)$ whose density function is given by
\begin{eqnarray}
\rho_{MP,y}=\frac{1}{2\pi xy}\sqrt{(b-x)(x-a)}\mathbf{1}_{[a,b]}(x),\label{1.000}
\end{eqnarray}
where $a=(1-\sqrt{y})^2, b=(1+\sqrt{y})^2$, 
$s_p(z)$ almost surely converges to the Stieltjes transform $s(z)$ of $F_{MP,y}(x)$.
 Here
\begin{eqnarray}
s(z)=\frac{1-y-z+\sqrt{(z-1-y)^2-4y}}{2yz}, \label{2.10.0}
\end{eqnarray}
where the square root is defined as the analytic extension of the positive square root of the positive numbers. Moreover, $s(z)$ satisfies the equation
\begin{eqnarray}
s(z)+\frac{1}{y+z-1+yzs(z)}=0. \label{2.10}
\end{eqnarray}

If we denote the $k$-th row of $Y$ by $\mathbf{y}_k^T$ and the remaining $(p-1)\times n$ matrix after deleting $\mathbf{y}_k^T$ by $Y^{(k)}$, one has
\begin{eqnarray*}
W=\left(
\begin{array}{ccc}
1 &\mathbf{y}_1^TY^{(1)T}\\
Y^{(1)}\mathbf{y}_1  &Y^{(1)}Y^{(1)T}
\end{array}
\right).
\end{eqnarray*}
By Schur's complement,
\begin{eqnarray}
G_{11}&=&\frac{1}{1-z-\mathbf{y}_1^T Y^{(1)T}(Y^{(1)}Y^{(1)T}-z)^{-1}Y^{(1)}\mathbf{y}_1}\nonumber\\
&=&\frac{1}{1-z-\mathbf{y}_1^TY^{(1)T}Y^{(1)}(Y^{(1)T}Y^{(1)}-z)^{-1}\mathbf{y}_1}.\label{2.1}
\end{eqnarray}
The formula of $G_{kk}$ is analogous. By (\ref{2.3}), we have the following lemma on the decomposition of $s_p(z)$:
\begin{lemma} \label{le.1.3}For the matrix $W$, we have
\begin{eqnarray*}
s_p(z)=\frac1p\sum_{k=1}^p\frac{1}{1-z-\mathbf{y}_k^TY^{(k)T}Y^{(k)}(Y^{(k)T}Y^{(k)}-z)^{-1}\mathbf{y}_k}.
\end{eqnarray*}
\end{lemma}

The last main tool we need comes from the probability theory, which is a concentration inequality for projections of random vectors. The details of the proof can also be found in \cite{TV1}.
\begin{lemma}\label{le.2.5}
Let $\mathcal{X}=(\xi_1,\cdots,\xi_n)^T\in\mathbb{C}^n$ be a random vector whose entries are independent with mean zero, variance 1, and are bounded in magnitude by $K$ almost surely for some $K$, where $K\geq 10(\mathbb{E}|\xi|^4+1)$. Let $H$ be a subspace of dimension $d$ and $\pi_H$ the orthogonal projection onto $H$. Then
\begin{eqnarray*}
\mathbb{P}(|||\pi_H(\mathcal{X})||-\sqrt{d}|\geq t)\leq 10\exp(-\frac{t^2}{10K^2}).
\end{eqnarray*}
In particular, one has
\begin{eqnarray*}
||\pi_H(\mathcal{X})||=\sqrt{d}+O(K\log n)
\end{eqnarray*}
with overwhelming probability.
\end{lemma}

\section{Main Technical Results}
In this section, we provide our main technical results: the local MP law for sample correlation matrices, and the delocalization property for the singular vectors. Both results will be proved under much weaker assumption than $\mathbf{C}_1$. We form them into the following two theorems.

Let us introduce more notation. For any interval $I\subset \mathbb{R}$, we use $N_I$ to denote the number of the eigenvalues of $W$ falling into $I$, and use $|I|$ to denote the length of $I$.
\begin{theorem}\label{th.2.1}(Local MP law).
 Assume that $p/n\rightarrow y$ with $0<y<1$. And $\{x_{ij}: 1\leq i\leq p, 1\leq j\leq n\}$ is a collection of independent (but not necessary identically distributed) random variables with mean zero and variance 1. If $|x_{ij}|\leq K$ almost surely for some $K=o(p^{1/C_0}\delta^2log^{-1}p)$ with some $0<\delta<1/2$ and some large constant $C_0$ for all $i,j$, one has with overwhelming probability that the number of eigenvalues $N_I$ for any interval $I\subset[a/2,2b]$ with $|I|\geq \frac{K^2\log^{7}p}{\delta^9p}$ obeys
\begin{eqnarray}
|N_I-p\int_{I}\rho_{MP,y}(x)dx|\leq\delta p|I|. \label{3.68}
\end{eqnarray}
\end{theorem}
\begin{remark}
The topic of the limiting spectrum distribution on short scales was firstly raised by Erd\H{o}s, Schlein and Yau in \cite{ESY1} for Wigner matrices. Such type of results are shown to be quite necessary for the proof of the famous universality conjectures in the Random Matrix Theory, for example, see \cite{ESYY} and \cite{TV1}.
\end{remark}
\begin{remark}
A strong type of the local MP law has been established for more general matrix models in a very recent paper of Pillai and Yin, see Theorem 1.5, \cite{PY}. In fact, from Theorem 1.5 of \cite{PY}, one can get a more precise bound than that in (\ref{3.68}) if we replace $\rho_{MP,y}(x)$ by the nonasymptotic MP law $\rho_{W}(x)$ defined in Section 4. Moreover, Pillai and Yin's strong local MP law also provides some crucial estimates on individual elements of the Green function $G$, which will be used to establish our Green function comparison theorem in Section 4.
\end{remark}
\begin{theorem} \label{th.2.2}(Delocalization of singular vectors)
Under the assumptions of Theorem \ref{th.2.1} and $\mathbb{E}x_{ij}^3=0$, if we assume $x_{ij}'s$ are continuous random variables, then with overwhelming probability all the left and right unit singular vectors of $W$ have all components uniformly of size at most $p^{-1/2}K^{C_0/2}\log^{O(1)}p$.
\end{theorem}
\begin{remark}\label{re.3.3.5}
Note that a little weaker delocalization property  for the left singular vector $v_i$ can also be found in Theorem 1.2 $(iv)$ of Pillai and Yin \cite{PY}.
\end{remark}

 Now if we denote
\begin{eqnarray*}
X=\left(
\begin{array}{cccc}
\frac{x_{11}}{\sqrt{n}} &\frac{x_{12}}{\sqrt{n}} &\cdots &\frac{x_{1n}}{\sqrt{n}}\\
\vdots &\vdots &\vdots &\vdots\\
\frac{x_{p1}}{\sqrt{n}} &\frac{x_{p2}}{\sqrt{n}} &\cdots &\frac{x_{pn}}{\sqrt{n}}
\end{array}
\right),
\end{eqnarray*}
then $S:=XX^T$ is the  sample covariance matrix corresponding to $W$. We further denote the ordered eigenvalues of $S$ by $0\leq\tilde{\lambda}_1\leq\cdots\leq\tilde{\lambda}_p$ and introduce the matrix
\begin{eqnarray*}
D=\left(
\begin{array}{cccc}
\frac{\sqrt{n}}{||\mathbf{x}_1||} &~ &~\\
~ &\ddots &~\\
~ &~ &\frac{n}{||\mathbf{x}_p||}
\end{array}
\right).
\end{eqnarray*}
By Theorem $5.9$ of \cite{BS}, we have $\tilde{\lambda}_p=b+o(1)$ holds with overwhelming probability. In fact, it is easy to see $\tilde{\lambda}_1=a+o(1)$ holds with overwhelming probability as well by a similar discussion through moment method. Observe that $W=DSD$, and $||D-I||_{op}=o(1)$ holds with overwhelming probability. By Lemma \ref{le.2.2}, we also have
\begin{eqnarray}
\lambda_1=a+o(1),~~~\lambda_p=b+o(1) \label{2.4}
\end{eqnarray}
holds with overwhelming probability. So below we always assume $\lambda_i\in(a/2,2b),1\leq i\leq p$.

The proof of Theorem \ref{th.2.1} is partially based on the lemmas of Section 2. It turns out to be quite similar to the case of sample covariance matrices and Wigner matrices, see \cite{ESY}, \cite{ESYY}, \cite{TV} and \cite{WK}. However, the delocalization of the right singular vector $u_i$ of $Y$ is an obstacle, owing to the lack of independence between the columns of $Y$.

For the convenience of the reader, we provide a short proof of Theorem \ref{th.2.1} at first.  Our main task in this section is the proof of Theorem \ref{th.2.2}, more precisely, the right singular vector part of the theorem.
\begin{proof}[Proof of Theorem \ref{th.2.1}]
We  provide the following crude upper bound on $N_I$ at first.
\begin{lemma}\label{le.2.4}Under the assumptions of Theorem \ref{th.2.1}, we have for any interval $I\subset\mathbb{R}$ with $|I|\gg K^2\log ^2p/p$, and large enough $C>0$
\begin{eqnarray*}
N_I\leq Cp|I|
\end{eqnarray*}
with overwhelming probability.
\end{lemma}
\begin{proof} Firstly we introduce the notation
\begin{eqnarray*}
W^{(k)}=Y^{(k)}Y^{(k)T},~~~~\mathcal{W}^{(k)}=Y^{(k)^T}Y^{(k)},~~~~G^{(k)}=(W^{(k)}-z)^{-1},~~~~\mathcal{G}^{(k)}=(\mathcal{W}^{(k)}-z)^{-1}.
\end{eqnarray*}
Let $\lambda_\alpha^{(1)}, \alpha=1,\cdots,p-1$ denote the eigenvalues of the $(p-1)\times(p-1)$ matrix $W^{(1)}$. Thus $\lambda_\alpha^{(1)}, \alpha=1,\cdots, p-1$ are also the eigenvalues of the $n\times n$ matrix $\mathcal{W}^{(1)}$, whose other eigenvalues are all zeros. We further use $\nu_\alpha$ to denote  the eigenvector of $\mathcal{W}^{(1)}$ corresponding to the eigenvalue $\lambda_\alpha^{(1)}$, and introduce the quantity
\begin{eqnarray}
\xi_\alpha=n|\mathbf{y}_1\cdot \nu_\alpha|^2=\frac{n}{||\mathbf{x}_1||^2}|\mathbf{x}_1\cdot \nu_\alpha|^2=:\frac{n}{||\mathbf{x}_1||^2}\tilde{\xi}_\alpha.\label{3.01}
\end{eqnarray}
We can rewrite (\ref{2.1}) as
\begin{eqnarray}
G_{11}=\frac{1}{1-z-\frac{1}{n}\sum_{\alpha=1}^{p-1}\frac{\lambda_\alpha^{(1)}\xi_\alpha}{\lambda_\alpha^{(1)}-z}}.\label{3.05}
\end{eqnarray}

By Cauchy's interlacing law, we also have $\lambda_\alpha^{(1)}\in[a/2,2b]$ with overwhelming probability. Then for any $z=E+i\eta$ such that $E\in [a/2,2b]$, we have
\begin{eqnarray}
\left|\Im G_{kk}\right|\leq\frac{1}{\eta+\frac{\eta}{n}\sum_{\alpha=1}^{p-1}\frac{\lambda_\alpha^{(1)}\xi_\alpha}{(\lambda_\alpha^{(1)}-E)^2+\eta^2}}
\leq \frac{C_1p\eta}{\sum_{\alpha:|\lambda_\alpha^{(1)}-E|\leq\eta}\xi_\alpha}\label{3.36}
\end{eqnarray}
 for any $k\in\{1,\cdots,p\}$. Now we set $I=[E-\eta/2,E+\eta/2]$. Notice that there always exists  some positive constant $C_2$ such that
\begin{eqnarray}
N_I\leq C_2p\eta\Im s_p(z)=C_2\eta\sum_{k=1}^p \Im G_{kk}.\label{3.36.5}
\end{eqnarray}
If we set $C_3=C_1C_2$, it follows from (\ref{3.36}) and (\ref{3.36.5}) that
\begin{eqnarray}
&&\mathbb{P}(N_I\geq Cp\eta)\nonumber\\
&=&\mathbb{P}\bigg{(}\sum_{k=1}^p{\Im}G_{kk}\geq C_2^{-1}Cp~{\rm{and}}~ N_I\geq Cp\eta\bigg{)}\nonumber\\
&\leq&p\mathbb{P}\bigg{(}\sum_{\alpha:|\lambda_\alpha^{(1)}-E|\leq \eta/2}\xi_\alpha\leq C_3C^{-1}p\eta ~{\rm{and}}~ N_I\geq Cp\eta\bigg{)}\nonumber\\
&\leq&p\mathbb{P}\bigg{(}\frac{n}{||\mathbf{x}_1||^2}\sum_{\alpha:|\lambda_\alpha^{(1)}-E|\leq \eta/2}\tilde{\xi}_\alpha\leq C_3C^{-1}p\eta ~{\rm{and}}~ N_I\geq Cp\eta\bigg{)}\nonumber\\
&\leq& p\mathbb{P}(||\mathbf{x}_1||^2\geq 2n)+p\mathbb{P}\bigg{(}\sum_{\alpha:|\lambda_\alpha^{(1)}-E|\leq \eta/2}\tilde{\xi}_\alpha\leq 2C_3C^{-1}p\eta ~{\rm{and}}~ N_I\geq Cp\eta\bigg{)}.
 \label{3.7}
\end{eqnarray}
The first term of (\ref{3.7}) is obviously exponential small by the Hoeffding inequality. For the second term, we use Lemma \ref{le.2.5}. Now we specialize $\mathcal{X}$ in Lemma \ref{le.2.5} to be $\mathbf{x}_1$ and the subspace $H$ to be the one generated by eigenvectors $\{\nu_\alpha: \lambda_\alpha^{(1)}\in I\}$. Thus one has $$d=N_I\geq Cp\eta\gg CK^2\log^2n.$$
Then by Lemma \ref{le.2.5} we have
\begin{eqnarray*}
\sum_{\alpha:|\lambda_\alpha^{(1)}-E|\leq \eta/2}\tilde{\xi}_\alpha=||\pi_H(\mathcal{X})||^2>\frac12Cp\eta
\end{eqnarray*}
 with overwhelming probability. This implies that the second term of (\ref{3.7}) is exponential small when $C$ is large enough. So we conclude the proof of Lemma \ref{le.2.4}.
\end{proof}

Now we proceed to prove Theorem \ref{th.2.1}. The basic strategy is to compare $s_p(z)$ and $s(z)$ with small imaginary part $\eta$. In fact, we have the following proposition.
\begin{prop} \label{pr.3.4}Let $1/10\geq\eta\geq\frac1n$, and $L_1, L_2,\epsilon,\delta>0$. Suppose that one has the bound
\begin{eqnarray*}
|s_p(z)-s(z)|\leq\delta
\end{eqnarray*}
with (uniformly) overwhelming probability for all $z=E+i\eta$ such that $E\in[L_1,L_2]$ and ${\Im z}\geq\eta$. Then for any interval $I\subset[L_1+\epsilon,L_2-\epsilon ]$ with $|I|\geq \max (2\eta,\frac{\eta}{\delta}\log\frac{1}{\delta})$, one has
\begin{eqnarray*}
|N_I-n\int_I\rho_{MP,y}(x)dx|\leq \delta n|I|
\end{eqnarray*}
with overwhelming probability.
\end{prop}
\begin{remark}
Proposition \ref{pr.3.4} is an extension of Lemma 29 of \cite{TV} up to the edge, whose proof can be found in \cite{WK}. In fact, the proof can be taken in the same manner as that of Lemma 64 in \cite{TV1} for the Wigner matrix.
\end{remark}

So in view of Proposition \ref{pr.3.4}, to prove Theorem \ref{th.2.1}, we only need to prove that
the bound
\begin{eqnarray}
|s_p(z)-s(z)|\leq\delta \label{3.36.6}
\end{eqnarray}
holds with (uniformly) overwhelming probability for all $z=E+i\eta$ such that $E\in[a/2-\epsilon,2b+\epsilon]$ and $1/10\geq\eta\geq \frac{K^2\log^6n}{n\delta^8}$.
To prove (\ref{3.36.6}) we need to derive a consistent equation for $s_p(z)$, which is similar to the equation (\ref{2.10}) for $s(z)$.

Firstly by Lemma \ref{le.1.3} we can rewrite $s_p(z)$ as
\begin{eqnarray*}
s_p(z)=\frac1p\sum_{k=1}^{p}\frac{1}{1-z-d_k},
\end{eqnarray*}
with
\begin{eqnarray*}
d_k=\mathbf{y}_k^T\mathcal{W}^{(k)}\mathcal{G}^{(k)}\mathbf{y}_k.
\end{eqnarray*}

Then the proof of (\ref{3.36.6}) can be taken in the same manner as the counterpart of the sample covariance matrix case (see the proof of formula (4.12) of \cite{WK}). We only state the different parts below and leave the details to the reader. We remark here that we consider the domain $[L_1,L_2]=[a/2-\epsilon,2b+\epsilon]$ rather than $[a,b]$ in \cite{WK}. However, if one goes through the proof in \cite{WK}, it is not difficult to see that the proof towards any domain $[L_1,L_2]$ containing $[a,b]$ is the same. The only minor difference between our case and the sample covariance matrix in \cite{WK} is the estimation of $d_k$.
We will only deal with $d_1$ in the sequel. The others are analogous. By (\ref{3.01}) and (\ref{3.05}), we have
\begin{eqnarray}
d_1=\frac1n\sum_{\alpha=1}^{p-1}\frac{\lambda_\alpha^{(1)}\xi_\alpha}{\lambda_\alpha^{(1)}-z}
=\frac1n\sum_{\alpha=1}^{p-1}\frac{\lambda_\alpha^{(1)}}{\lambda_\alpha^{(1)}-z}+\frac1n
\sum_{\alpha=1}^{p-1}\frac{\lambda_\alpha^{(1)}(\xi_\alpha-1)}{\lambda_\alpha^{(1)}-z}.\label{3.10.1}
\end{eqnarray}
For the first term of (\ref{3.10.1}) we have
\begin{eqnarray*}
\frac1n\sum_{\alpha=1}^{p-1}\frac{\lambda_\alpha^{(1)}}{\lambda_\alpha^{(1)}-z}=\frac{p-1}{n}+\frac{z}{n}\sum_{j=1}^{p-1}\frac{1}{\lambda_\alpha^{(1)}-z}
:=\frac{p-1}{n}(1+zs^{(1)}_p(z)),
\end{eqnarray*}
where
\begin{eqnarray*}
s^{(1)}_p(z)=\frac{1}{p-1}\sum_{j=1}^{p-1}\frac{1}{\lambda_\alpha^{(1)}-z}
\end{eqnarray*}
is the Stieltjes transform of the ESD of $W^{(1)}$. Then by the Cauchy's interlacing property, we have
\begin{eqnarray*}
|s_p(z)-(1-\frac1p)s^{(1)}_p(z)|=O(\frac1p\int_{\mathbb{R}}\frac{1}{|x-z|^2}dx)=O(\frac{1}{p\eta}).
\end{eqnarray*}
Consequently one has
\begin{eqnarray}
\frac1n\sum_{\alpha=1}^{p-1}\frac{\lambda_\alpha^{(1)}}{\lambda_\alpha^{(1)}-z}=\frac{p-1}{n}+z\frac{p}{n}s_p(z)+o(\delta^2). \label{3.10.11}
\end{eqnarray}
Now we provide the following lemma on the second term of (\ref{3.10.1}).
\begin{lemma} \label{le.3.5.5} For all $z=E+i\eta$ with $E\in[a/2-\epsilon,2b+\epsilon]$ and $\eta\geq \frac{K^2\log^6n}{n\delta^8}$,
\begin{eqnarray*}
\frac1n\sum_{\alpha=1}^{p-1}\frac{\lambda_\alpha^{(1)}(\xi_\alpha-1)}{\lambda_\alpha^{(1)}-z}=o(\delta^2)
\end{eqnarray*}
uniformly in $z$ with overwhelming probability.
\end{lemma}
\begin{proof}
We set $R_j=(\xi_j-1)$. By (\ref{3.01}) and the fact that
\begin{eqnarray}
\frac{n}{||\mathbf{x}_1||^2}=1+O(\frac{K^2\log^2n}{\sqrt{n}})\label{3.11.1}
\end{eqnarray}
holds with overwhelming probability, we have for any $T\subset\{1,\cdots,p-1\}$
\begin{eqnarray}
\sum_{j\in T}R_j=\frac{n}{||\mathbf{x}_1||^2}\sum_{j\in T}|\mathbf{x}_1\cdot \nu_j|^2-|T|.\label{3.11.2}
\end{eqnarray}
By using Lemma \ref{le.2.5}, we have
\begin{eqnarray}
\sum_{j\in T}|\mathbf{x}_1\cdot \nu_j|^2=T+O\left(\sqrt{T}K\log n\vee K^2\log^2n\right),\label{3.12.1}
\end{eqnarray}
where $a\vee b=\max(a,b)$. By inserting (\ref{3.11.1}) and (\ref{3.12.1}) into (\ref{3.11.2}), we have
\begin{eqnarray*}
\sum_{j\in T}R_j=\sum_{j\in T}|\mathbf{x}_1\cdot \nu_j|^2-|T|+O\left(\frac{TK^{4}\log^{4} n}{\sqrt{n}}\right).
\end{eqnarray*}
If we choose $T=\log^{O(1)}n$, we always have
\begin{eqnarray*}
\sum_{j\in T}R_j=\sum_{j\in T}|\mathbf{x}_1\cdot \nu_j|^2-|T|+o(\delta^2).
\end{eqnarray*}
Then the following part of the proof is the same as that in the sample covariance matrix case. One can refer to the proof of Proposition 4.6 of \cite{WK} for details.
\end{proof}
Now we proceed to the proof of Theorem \ref{th.2.1}. By (\ref{3.10.1}), (\ref{3.10.11}) and Lemma \ref{le.3.5.5} we can get the following equation
\begin{eqnarray}
s_p(z)+\frac{1}{\frac pn+z-1+z\frac pn s_p(z)+o(\delta^2)}=0. \label{3.10.12}
\end{eqnarray}
By a standard comparison of (\ref{3.10.12}) and (\ref{2.10}) (see \cite{WK} for example), we have (\ref{3.36.6}). Thus by Proposition \ref{pr.3.4} we conclude the proof of Theorem \ref{th.2.1}.
\end{proof}

Now we turn to the proof of Theorem \ref{th.2.2}. At first, we introduce the matrix $\widehat{W}_{(n)}:=\widehat{Y}_{(n)}\widehat{Y}_{(n)}^T$ with
\begin{eqnarray*}
\widehat{Y}_{(n)}=\left(
\begin{array}{ccccc}
\frac{x_{11}}{||\widehat{\mathbf{x}}_1||} &\frac{x_{12}}{||\widehat{\mathbf{x}}_1||} &\cdots &\frac{x_{1,n-1}}{||\widehat{\mathbf{x}}_1||}\\
\frac{x_{21}}{||\widehat{\mathbf{x}}_2||} &\frac{x_{22}}{||\widehat{\mathbf{x}}_2||} &\cdots &\frac{x_{2,n-1}}{||\widehat{\mathbf{x}}_2||}\\
\cdots &\cdots &\cdots &\cdots\\
\frac{x_{p1}}{||\widehat{\mathbf{x}}_p||} &\frac{x_{p2}}{||\widehat{\mathbf{x}}_p||} &\cdots &\frac{x_{p,n-1}}{||\widehat{\mathbf{x}}_p||}
\end{array}
\right),
\end{eqnarray*}
where
\begin{eqnarray*}
\widehat{\mathbf{x}}_j=(x_{j1},x_{j2},\cdots,x_{j,n-1})^T.
\end{eqnarray*}

We will need the following lemma on eigenvalue collision.
\begin{lemma} \label{le.rrr}
If we assume the random variables $x_{ij}'s$ are continuous, we have the following events hold with probability one.\\
$i)$: $W$ has simple eigenvalues, i.e. $\lambda_1<\lambda_2<\cdots<\lambda_p$.\\
$ii)$: $W$ and $W^{(p)}$ have no eigenvalue in common.\\
$iii)$: $W$ and $\widehat{W}_{(n)}$ have no eigenvalue in common.
\end{lemma}

The proof of Lemma \ref{le.rrr} will be postponed to Appendix A.
\begin{proof}[Proof of Theorem \ref{th.2.2}] The proof for the left singular vectors is nearly the same as the sample covariance matrix case shown in \cite{WK} by using Lemma \ref{le.2.0}, $ii)$ of Lemma \ref{le.rrr} and Theorem \ref{th.2.1}. Moreover, as we have mentioned in the Remark \ref{re.3.3.5}, a slightly weaker delocalization property for the left singular vectors has been provided in \cite{PY}. So we will only present the proof for the right singular vectors below.

Below we denote the $k$-th column of $Y$ by $h_k$, and the remaining $p\times (n-1)$ matrix after deleting $h_k$ by $Y_{(k)}$. Note that $Y_{(n)}$ is not independent of the last column $h_n$. However, for the sample covariance matrix case, the independence between the column and the corresponding submatrix is essential for one to use the concentration results such as Lemma \ref{le.2.5}. To overcome the inconvenience caused by the dependence, we will use the modified matrix $\widehat{Y}_{(n)}$ defined above. Notice that the matrix $\widehat{Y}_{(n)}$ is independent of the random vector $(x_{1n},x_{2n}\cdots,x_{pn})^T$.

Now we define
\begin{eqnarray*}
\Delta_1=Y^T_{(n)}Y_{(n)}-\widehat{Y}^T_{(n)}\widehat{Y}_{(n)},
\end{eqnarray*}
and
\begin{equation}\label{a1}
\Delta_2=Y^T_{(n)}-\widehat{Y}^T_{(n)}.
\end{equation}
The following lemma handles the operator norms of $\Delta_1$ and $\Delta_2$.
\begin{lemma}\label{le.3.4} Under the assumption of Theorem \ref{th.2.2}, we have
\begin{eqnarray*}
||\Delta_1||_{op},||\Delta_2||_{op}= O(\frac{K^2}{n})
\end{eqnarray*}
 with overwhelming probability.
\end{lemma}
\begin{proof}
Observe that
\begin{eqnarray*}
\Delta_1=(Y^T_{(n)}-\widehat{Y}^T_{(n)})Y_{(n)}+\widehat{Y}^T_{(n)}(Y_{(n)}-\widehat{Y}_{(n)})=\Delta_2Y_{(n)}+\widehat{Y}^T_{(n)}\Delta_2^T.
\end{eqnarray*}
We only discuss the second term since the first one is analogous. It is easy to see the entries of $\Delta_2^T$ satisfy
\begin{eqnarray*}
\frac{x_{ij}}{||\mathbf{x}_i||}-\frac{x_{ij}}{||\widehat{\mathbf{x}}_i||}=\frac{x_{ij}(||\widehat{\mathbf{x}}_i||^2
-||\mathbf{x}_i||^2)}{||\mathbf{x}_i||||\widehat{\mathbf{x}}_i||(||\mathbf{x}_i||+||\widehat{\mathbf{x}}_i||)}
=-\frac{x_{in}^2}{||\mathbf{x}_i||(||\mathbf{x}_i||+||\widehat{\mathbf{x}}_i||)}\cdot\frac{x_{ij}}{||\widehat{\mathbf{x}_i}||}.
\end{eqnarray*}
It follows that
\begin{eqnarray*}
\widehat{Y}^T_{(n)}\Delta_2^T:=-\widehat{Y}^T_{(n)}\Delta_3\widehat{Y}_{(n)},
\end{eqnarray*}
where $\Delta_3$ is a $p\times p$ diagonal matrix with $(i,i)$-th entry to be
\begin{eqnarray*}
\frac{x_{in}^2}{||\mathbf{x}_i||(||\mathbf{x}_i||+||\widehat{\mathbf{x}}_i||)}.
\end{eqnarray*}
Thus it is easy to see
\begin{eqnarray*}
||\Delta_3||_{op}=O(\frac{K^2}{n}),
\end{eqnarray*}
with overwhelming probability. Together with the fact that $||\widehat{Y}_{(n)}||_{op}\leq C$ holds with overwhelming probability, we can conclude the proof of Lemma \ref{le.3.4}.
\end{proof}

Now we proceed to the proof of Theorem \ref{th.2.2}. If we denote
\begin{eqnarray*}
u_i=\binom{\mathbf{w}}{x},
\end{eqnarray*}
where $x$ is the last component of $u_i$. Without loss of generality, we can only prove the theorem for $x$. Notice that $u_i$ is the eigenvector of $\mathcal{W}=Y^TY$ corresponding to the eigenvalue $\lambda_i$.
From
 \begin{eqnarray*}
 \left(
 \begin{array}{cccc}
 Y^T_{(n)}Y_{(n)} &Y^T_{(n)}h_n\\\\
 h_n^TY_{(n)}  &h_n^Th_n
 \end{array}
 \right)\binom{\mathbf{w}}{x}=\lambda_i\binom{\mathbf{w}}{x},
 \end{eqnarray*}
we have
 \begin{eqnarray}
 Y^T_{(n)}Y_{(n)}\mathbf{w}+xY^T_{(n)}h_n=\lambda_i\mathbf{w},\label{3.51}
 \end{eqnarray}
 and
 \begin{eqnarray}
 h_n^TY_{(n)}\mathbf{w}+xh_n^Th_n=\lambda_ix. \label{3.52}
 \end{eqnarray}
 (\ref{3.51}) can be rewritten as
 \begin{eqnarray*}
 (\widehat{Y}^T_{(n)}\widehat{Y}_{(n)}+\Delta_1)\mathbf{w}+x(\widehat{Y}^T_{(n)}+\Delta_2)h_n=\lambda_i\mathbf{w}.
 \end{eqnarray*}
It follows that
 \begin{eqnarray}
 (\widehat{Y}^T_{(n)}\widehat{Y}_{(n)}-\lambda_i)\mathbf{w}=-x\widehat{Y}^T_{(n)}h_n-x\Delta_2h_n-\Delta_1\mathbf{w}.\label{3.59}
 \end{eqnarray}
 Note that $\widehat{Y}_{(n)}^T\widehat{Y}_{(n)}$ share the same nonzero eigenvalues with $\widehat{W}_{(n)}$, so by $iii)$ of Lemma \ref{le.rrr}, we can always view that the matrix $\widehat{Y}^T_{(n)}\widehat{Y}_{(n)}-\lambda_i$ is invertible. Consequently,
 \begin{eqnarray*}
 ||\mathbf{w}||^2
 &=&[x\widehat{Y}^T_{(n)}h_n+x\Delta_2h_n+\Delta_1\mathbf{w}]^T(\widehat{Y}^T_{(n)}\widehat{Y}_{(n)}-\lambda_i)^{-2}[x\widehat{Y}^T_{(n)}
 h_n+x\Delta_2h_n+\Delta_1\mathbf{w}]\nonumber\\
 \end{eqnarray*}
If $x=0$ then Theorem \ref{th.2.2} is evidently true. Consider $x\neq 0$ below. Together with the fact that $x^2=1-||\mathbf{w}||^2$, we have
\begin{eqnarray}
 x^2
=\frac{1}{1+[\widehat{Y}_{(n)}^T h_n+\Delta_2 h_n+x^{-1}\Delta_1 \mathbf{w}]^T(\widehat{Y}_{(n)}^T\widehat{Y}_{(n)}-\lambda_i)^{-2}[\widehat{Y}_{(n)}^T h_n+\Delta_2 h_n+x^{-1}\Delta_1 \mathbf{w}]}.\nonumber
\end{eqnarray}

Now if we use $\hat{\lambda}_j$ to denote the ordered nonzero eigenvalue of $\widehat{Y}_{(n)}^T\widehat{Y}_{(n)}$ and  $\hat{u}_j$  the corresponding unit eigenvector. And set the projection
\begin{eqnarray*}
  \widehat{P}=I-\sum_{j=1}^p\hat{u}_j\hat{u}_j^T.
  \end{eqnarray*}
  Then by the spectral decomposition one has
\begin{eqnarray}
x^2=\frac{1}{1+\sum_{j=1}^p\frac{1}{(\hat{\lambda}_j-\lambda_i)^2}|\hat{u}_j\cdot (\widehat{Y}_{(n)}^T h_n+\Delta_2 h_n+x^{-1}\Delta_1 \mathbf{w})|^2+\Delta},
\label{3.58}
\end{eqnarray}
where
  \begin{eqnarray*}
  \Delta=\frac{1}{\lambda_i^2}||\widehat{P}(\widehat{Y}_{(n)}^T h_n+\Delta_2 h_n+x^{-1}\Delta_1 \mathbf{w})||^2.
  \end{eqnarray*}

 Therefore to show $|x|\leq n^{-1/2}K^{C_0/2}\log^{O(1)}n$, we only need to prove
  \begin{eqnarray}
\quad\quad  \sum_{j=1}^p\frac{1}{(\hat{\lambda}_j-\lambda_i)^2}|\hat{u}_j\cdot (\widehat{Y}_{(n)}^T h_n+\Delta_2 h_n+x^{-1}\Delta_1 \mathbf{w})|^2\geq nK^{-C_0}\log^{-O(1)}n. \label{3.53}
  \end{eqnarray}
To prove (\ref{3.53}), we need to separate the issue into the bulk case and the edge case. Before that, we shall provide the following lemma which will be used in both cases.
\begin{lemma} \label{le.3.5} If we denote the unit eigenvector of $\widehat{W}_{(n)}$ corresponding to $\hat{\lambda}_j$ by $\hat{v}_j$,  under the assumption of Theorem \ref{th.2.2} we have for any $J\subseteq \{1,\cdots,p\}$ with $|J|=d\leq nK^{-3}$,
\begin{eqnarray*}
\sqrt{n}[\sum_{j\in J}(\hat{v}_{j}\cdot h_n)^2]^{1/2}=\sqrt{d}+O(K\log n)
\end{eqnarray*}
with overwhelming probability.
\end{lemma}
We will postpone the proof of Lemma \ref{le.3.5} to Appendix B. In fact, it can be viewed as a modification of Lemma \ref{le.2.5}. \\

Now we decompose the proof of Theorem \ref{th.2.2} into two parts: bulk case and edge case.\\\\
$\bullet$ Bulk case: \emph{$\lambda_i\in[a+\epsilon, b-\epsilon]$ for some $\epsilon>0$}\\

Note that the local MP law (Theorem \ref{th.2.1}) can also be applied to
the matrix $\widehat{Y}^T_{(n)}\widehat{Y}_{(n)}$. Thus we can find a set $J\subseteq\{1,\cdots,p\}$ with $|J|\geq K^2\log^{20}n$ such that $\hat{\lambda}_j=\lambda_i+O(K^2\log^{20}n/n)$ for any $j\in J$ when $\lambda_i$ is in the bulk region of the MP law.
It follows that
\begin{eqnarray}
&&\sum_{j\in J}\frac{1}{(\hat{\lambda}_j-\lambda_i)^2}|\hat{u}_j\cdot (\widehat{Y}_{(n)}^T h_n+\Delta_2 h_n+x^{-1}\Delta_1 \mathbf{w})|^2\nonumber\\
&&\geq C\frac{n^2}{K^4\log^{40}n}\sum_{j\in J}|\hat{u}_j\cdot (\widehat{Y}_{(n)}^T h_n+\Delta_2 h_n+x^{-1}\Delta_1 \mathbf{w})|^2. \label{3.67}
\end{eqnarray}

 By the singular value decomposition, we have
 \begin{equation}\label{a2}
 \hat{u}_j\cdot \widehat{Y}^T_{(n)}h_n=\hat{\lambda}_j^{1/2}\hat{v}_j\cdot h_n.
 \end{equation}
Now we compare
\begin{eqnarray}
 \sum_{j\in J}|\hat{u}_j\cdot \widehat{Y}^T_{(n)}h_n|^2=\sum_{j\in J}\hat{\lambda}_j|\hat{v}_j\cdot h_n|^2\label{3.54}
\end{eqnarray}
with
\begin{eqnarray}
\sum_{j\in J}|\hat{u}_j\cdot(\Delta_2 h_n+x^{-1}\Delta_1\mathbf{w})|^2\label{3.55}
\end{eqnarray}
for any $J\subset\{1,\cdots,p\}$ such that $K^2\log^{20}n\leq |J|\leq nK^{-3}$.

If $|x|\leq n^{-1/2}K^{C_0/2}\log^{O(1)}n$, then we get the conclusion for the bulk case. So we assume $|x|\geq n^{-1/2}K^{C_0/2}\log^{O(1)}n$ below to get (\ref{3.53}). By Lemma \ref{le.3.4}, if we choose $C_0\geq 20$ (say), we have
\begin{eqnarray}
(\ref{3.55})&\leq  &
2|J|(||\Delta_2||_{op}||h_n||)^2+2x^{-2}|J|(||\Delta_1||_{op}||\mathbf{w}||)^2 \nonumber\\
&\leq& |J|n^{-1}K^{-C_0/2}\log^{-O(1)}n
\label{3.60}
\end{eqnarray}
 with overwhelming probability.
On the other side, Lemma \ref{le.3.5} implies
\begin{eqnarray}
(\ref{3.54})=Cn^{-1}(|J|+O(K^2\log^2 n))\label{3.56}
\end{eqnarray}
with overwhelming probability.
So one has
\begin{eqnarray}
\sum_{j\in J}|\hat{u}_j\cdot \widehat{Y}^T_{(n)}h_n|^2\gg \sum_{j\in J}|\hat{u}_j\cdot(\Delta_2 h_n+x^{-1}\Delta_1\mathbf{w})|^2,\label{3.57}
\end{eqnarray}
where $\gg$ means ``much larger than", i.e. $$\Big(\sum_{j\in J}|\hat{u}_j\cdot(\Delta_2 h_n+x^{-1}\Delta_1\mathbf{w})|^2\Big)/\Big(\sum_{j\in J}|\hat{u}_j\cdot \widehat{Y}^T_{(n)}h_n|^2\Big)=o(1).$$

Notice that for any real number sequence $\{S_1,\cdots,S_m\}$ and $\{T_1,\cdots,T_m\}$ with $\sum_{i=1}^mS_i^2\gg \sum_{i=1}^{m}T_i^2$, there exists some $c$ near 1 such that $\sum_{i=1}^m(S_i+T_i)^2\geq c\sum_{i=1}^mS_i^2$. Therefore by (\ref{3.56}),(\ref{3.57}) and (\ref{3.67}) we can obtain
\begin{eqnarray*}
\sum_{j\in J}\frac{1}{(\hat{\lambda}_j-\lambda_i)^2}|\hat{u}_j\cdot (\widehat{Y}_{(n)}^T h_n+\Delta_2 h_n+x^{-1}\Delta_1 \mathbf{w})|^2\geq CnK^{-2}\log^{-20}n,
\end{eqnarray*}
which implies (\ref{3.53}) directly. So we conclude the proof for the bulk case.\\

Next, we turn to the edge case.\\\\
$\bullet$ Edge case: $a-o(1)\leq\lambda_i\leq a+\epsilon$ or $b-\epsilon\leq\lambda_i\leq b+o(1)$ with some $\epsilon>0$.\\

For the edge case we also begin with the representation (\ref{3.58}).
By (\ref{3.59}), we have
\begin{eqnarray}
\mathbf{w}=-x(\widehat{Y}^T_{(n)}\widehat{Y}_{(n)}-\lambda_i)^{-1}(\widehat{Y}^T_{(n)}h_n+\Delta_2 h_n+x^{-1}\Delta_1 \mathbf{w}). \label{3.100}
\end{eqnarray}
Inserting (\ref{3.100}) and (\ref{a1}) into (\ref{3.52}) we find
\begin{eqnarray*}
(\widehat{Y}^T_{(n)}h_n+\Delta_2 h_n)^T(\widehat{Y}^T_{(n)}\widehat{Y}_{(n)}-\lambda_i)^{-1}(\widehat{Y}^T_{(n)}h_n+\Delta_2 h_n+x^{-1}\Delta_1 \mathbf{w})=h_n^Th_n-\lambda_i.
\end{eqnarray*}
Furthermore,
\begin{eqnarray*}
&&|x^{-1}\mathbf{w}^T\Delta_1^T(\widehat{Y}^T_{(n)}\widehat{Y}_{(n)}-\lambda_i)^{-1}(\widehat{Y}^T_{(n)}h_n+\Delta_2 h_n+x^{-1}\Delta_1 \mathbf{w})|\nonumber\\
&&=|x^{-2}\mathbf{w}^T\Delta_1^T\mathbf{w}|\leq |x|^{-2}||\Delta_1||_{op}||\mathbf{w}||^2=||\Delta_1||_{op}\frac{1-x^2}{x^2}.
\end{eqnarray*}
Thus one has
\begin{eqnarray}
&&(\widehat{Y}^T_{(n)}h_n+\Delta_2h_n+x^{-1}\Delta_1 \mathbf{w})^T(\widehat{Y}^T_{(n)}\widehat{Y}_{(n)}-\lambda_i)^{-1}(\widehat{Y}^T_{(n)}h_n+\Delta_2 h_n+x^{-1}\Delta_1 \mathbf{w})\nonumber\\
&&=h_n^Th_n-\lambda_i+O\left(||\Delta_1||_{op}\frac{1-x^2}{x^2}\right). \label{3.49}
\end{eqnarray}

Similarly to the bulk case, we only need to get (\ref{3.53}). Below we also assume $|x|\geq C n^{-1/2}K^{C_0/2}\log^{O(1)}n$ to get (\ref{3.53}).
Similar to (\ref{3.60}), by using Lemma \ref{le.3.4} we have
\begin{eqnarray}
|\hat{u}_j\cdot (\Delta_2 h_n+x^{-1}\Delta_1 \mathbf{w})|^2\leq n^{-1}K^{-C_0/2}\log^{-O(1)}n.\label{3.61}
\end{eqnarray}
Moreover, by Lemma \ref{le.3.5} and (\ref{a2}), we also have
\begin{eqnarray}
|\hat{u}_j\cdot \widehat{Y}_{(n)}^T h_n|^2=\hat{\lambda}_i|\hat{v}_j\cdot h_n|\leq Cn^{-1}K^2\log^2 n \label{3.69}
\end{eqnarray}
holds with overwhelming probability.
Thus to provide (\ref{3.53}), it suffices to show
\begin{eqnarray*}
\sum_{j=1}^p\frac{1}{(\hat{\lambda}_j-\lambda_i)^2}|\hat{u}_j\cdot (\widehat{Y}_{(n)}^T h_n+\Delta_2 h_n+x^{-1}\Delta_1 \mathbf{w})|^4\geq
K^{-C_0+2}\log^{-O(1)}n
\end{eqnarray*}
instead. By the Cauchy-Schwarz inequality, we only need to prove
\begin{eqnarray}
\sum_{i-T_{-}\leq j\leq i+T_{+}}\frac{1}{|\hat{\lambda}_j-\lambda_i|}|\hat{u}_j\cdot (\widehat{Y}_{(n)}^T h_n+\Delta_2 h_n+x^{-1}\Delta_1 \mathbf{w})|^2 \geq \log^{-O(1)}n\label{3.66}
\end{eqnarray}
with overwhelming probability for some $1\leq T_{-}<T_{+}\leq K^{2}\log^{O(1)}n$.

Notice that under the assumption $|x|\geq C n^{-1/2}K^{C_0/2}\log^{O(1)}n$, by Lemma \ref{le.3.4} we have
\begin{eqnarray*}
||\Delta_1||_{op}\frac{1-x^2}{x^2}=o(1).
\end{eqnarray*}
Moreover, it is not difficult to see $h_n^Th_n=y+o(1)$ with overwhelming probability. Thus by (\ref{3.49}), we have with overwhelming probability
$$\sum_{j=1}^p\frac{1}{(\hat{\lambda}_j-\lambda_i)}|\hat{u}_j\cdot (\widehat{Y}_{(n)}^T h_n+\Delta_2 h_n+x^{-1}\Delta_1 \mathbf{w})|^2-\frac{1}{\lambda_i}||\widehat{P}(\widehat{Y}_{(n)}^T h_n+\Delta_2 h_n+x^{-1}\Delta_1 \mathbf{w})||^2$$
$$
=h_n^Th_n-\lambda_i+O(||\Delta_1||_{op}\frac{1-x^2}{x^2})=y-\lambda_i+o(1).
$$
Observing that
\begin{eqnarray*}
\widehat{P}\widehat{Y}_{(n)}^T h_n=0
\end{eqnarray*}
and
\begin{eqnarray*}
||\widehat{P}(\Delta_2 h_n+x^{-1}\Delta_1 \mathbf{w})||^2\ll n^{-1},
\end{eqnarray*}
we also have
\begin{eqnarray}
\sum_{j=1}^p\frac{1}{(\hat{\lambda}_j-\lambda_i)}|\hat{u}_j\cdot (\widehat{Y}_{(n)}^T h_n+\Delta_2 h_n+x^{-1}\Delta_1 \mathbf{w})|^2=y-\lambda_i+o(1).\label{3.65}
\end{eqnarray}

So to prove (\ref{3.66}) we only need to evaluate
\begin{eqnarray}
\sum_{j<i-T_{-}~or~j>i+ T_{+}}\frac{1}{(\hat{\lambda}_j-\lambda_i)}|\hat{u}_j\cdot (\widehat{Y}_{(n)}^T h_n+\Delta_2 h_n+x^{-1}\Delta_1 \mathbf{w})|^2.\label{3.63}
\end{eqnarray}

To do this, we let $A>100$ be a constant large enough. For any interval $I$ of length $|I|=K^{2}\log^An/n$, we set $d_I:=\frac{\mathrm{dist}(\lambda_i,I)}{|I|}$, where
 $$\mathrm{dist}(\lambda_i,I)=\min_{x\in I}|\lambda_i-x|\mathrm{sgn}(\lambda_i,I).$$
 Here $\mathrm{sgn}(\lambda_i,I)=1$(\emph{resp.} $-1$) when $\lambda_i$ is on the left (\emph{resp.} right) hand side of $I$.

 By Theorem \ref{th.2.1}, the interval $I$ with $|d_I|<\log n$ contains at most $K^{2}\log^{O(1)}n$ eigenvalues. So we can set $T_{-},T_{+}$ accordingly so that such intervals don't contain any $\hat \lambda_j$ if $j<i-T_{-}$ or $j>i+T_{+}$. In the following we only consider $I$ such that $|d_I| \geq\log n$ in the estimation of \eqref{3.63}. Note that for $\hat{\lambda}_j\in I$,
\begin{eqnarray*}
\frac{1}{\hat{\lambda}_j-\lambda_i}=\frac{1}{d_I|I|}+O(\frac{1}{d_I^2|I|}).
\end{eqnarray*}
Using (\ref{3.61}) and (\ref{3.69}) again one has
\begin{eqnarray*}
&&2|(\hat{u}_j\cdot\widehat{Y}^T_{(n)}h_n)(\hat{u}_j\cdot(\Delta_2h_n+x^{-1}\Delta_1\mathbf{w}))|+|\hat{u}_j\cdot\Delta_2h_n+x^{-1}\hat{u}_j\cdot\Delta_1\mathbf{w}|^2\\
&\leq& Cn^{-1}K^{-O(1)}\log^{-O(1)}n
\end{eqnarray*}
when $|x|\geq n^{-1/2}K^{C_0/2}\log^{O(1)}n$. Thus we can find
$$
\sum_{j\in I}\frac{1}{|\hat{\lambda}_j-\lambda_i|}(2|
(\hat{u}_j\cdot\widehat{Y}^T_{(n)}h_n)(\hat{u}_j\cdot(\Delta_2h_n+x^{-1}\Delta_1\mathbf{w}))|+|\hat{u}_j\cdot\Delta_2h_n+x^{-1}\hat{u}_j\cdot\Delta_1\mathbf{w}|^2)$$
\begin{eqnarray}\leq C\frac{N_{I}}{|d_I||I|n}K^{-O(1)}\log^{-O(1)}n\leq C\frac{1}{|d_I|}K^{-O(1)}\log^{-O(1)}n.  \label{3.62}
\end{eqnarray}
Here we used Lemma \ref{le.2.4} in the last inequality.
Now we partition the real line into intervals $I$ of length $K^{2}\log^{A}n/n$, and sum (\ref{3.62}) over all intervals $I$ with $|d_I|\geq \log n$. Then
\begin{eqnarray*}
\sum_{I}\frac{1}{d_I}K^{-O(1)}\log^{-O(1)}n=o(1).
\end{eqnarray*}
So we can evaluate
\begin{eqnarray}
 \sum_{j<i-T_{-}~or~j>i+ T_{+}}\frac{1}{(\hat{\lambda}_j-\lambda_i)}|\hat{u}_j\cdot\widehat{Y}_{(n)}^T h_n|^2=\sum_{j<i-T_{-}~or~j>i+ T_{+}}\frac{\hat{\lambda}_j}{(\hat{\lambda}_j-\lambda_i)}|\hat{v}_j\cdot h_n|^2.
\label{3.64}
\end{eqnarray}
instead of (\ref{3.63}). The evaluation of (\ref{3.64}) is really the same as the counterpart in the sample covariance matrix case (see (4.5) in \cite{WK}) by inserting Lemma \ref{le.3.5}, so we omit the details here. In fact, we can finally get
\begin{eqnarray*}
\sum_{j<i-T_{-}~or~j>i+ T_{+}}\frac{\hat{\lambda}_j}{(\hat{\lambda}_j-\lambda_i)}|\hat{v}_j\cdot h_n|^2
&=&p.v.\int_a^b y\frac{x}{x-\lambda_i}\rho_{MP,y}(x)dx+o(1)\nonumber\\
&=&y+\lambda_i~p.v.\int_a^b\frac{\rho_{MP,y}(x)}{x-\lambda_i}dx+o(1)
\end{eqnarray*}
where $p.v.$ means the principal value.

Using the formula for the Stieltjes transform $s(z)$, one can get from residue calculus that for $\lambda_i\in[a,b]$,
\begin{eqnarray*}
p.v.\int_a^b\frac{\rho_{MP,y}(x)}{x-\lambda_i}dx=\frac{1-y-\lambda_i}{2y\lambda_i},
\end{eqnarray*}
and for $\lambda_i\not\in[a,b]$
\begin{eqnarray*}
p.v.\int_a^b\frac{\rho_{MP,y}(x)}{x-\lambda_i}dx=\frac{1-y-\lambda_i+\sqrt{(\lambda_i-1-y)^2-4y}}{2y\lambda_i}.
\end{eqnarray*}
Consequently by the definition of $a$ and $b$, if $|\lambda_i-a|\leq o(1)$, we have
\begin{eqnarray*}
(\ref{3.65})=-1+2\sqrt{y}+o(1),~~~(\ref{3.64})=\sqrt{y}+o(1).
\end{eqnarray*}
And if $|\lambda_i-b|\leq o(1)$, we have
\begin{eqnarray*}
(\ref{3.65})=-1-2\sqrt{y}+o(1),~~~(\ref{3.64})=-\sqrt{y}+o(1).
\end{eqnarray*}

Then it is easy to see when $0< y<1$, (\ref{3.66}) holds with overwhelming probability for the case where $|\lambda_i-a|=o(1)$ or $|\lambda_i-b|=o(1)$. Moreover by continuity we can adjust the value of $\epsilon$ to get the conclusion for the general case $a-o(1)\leq\lambda_i\leq a+\epsilon$ or $b-\epsilon\leq\lambda_i\leq b+o(1)$. Thus we complete the proof of the delocalization for $u_i$.
\end{proof}

\section{Green function comparison theorem}
In this section, we provide a Green function comparison theorem for the sample correlation matrices satisfying $\mathbf{C}_1$. The proof heavily relies on the recent results of Pillai and Yin \cite{PY} on sample covariance matrices and the delocalization property for the right singular vectors proved in the last section. At first, we will borrow some results from \cite{PY} directly with only minor notation change. In fact, by Theorem 1.5 in \cite{PY}, it is not difficult to see Theorem 1.2 and Theorem 1.3 of \cite{PY} also hold for sample correlation matrices under our basic condition $\mathbf{C_1}$.

 To state the results in \cite{PY}, we need to introduce some notation. Define the parameter
 \begin{eqnarray*}
 \varphi:=(\log p)^{\log\log p},
 \end{eqnarray*}
 and
 \begin{eqnarray*}
 \lambda_\pm:=(1\pm(\frac pn)^{1/2})^2.
 \end{eqnarray*}
 Moreover we introduce the ``nonasymptotic Marchenko-Pastur law ''
 \begin{eqnarray*}
 \rho_W(x)=\frac{n}{2\pi xp}\sqrt{(\lambda_+-x)(x-\lambda_-)}\mathbf{1}_{[\lambda_-,\lambda_+]}(x)
 \end{eqnarray*}
 and the corresponding distribution function $F_W(x)$
 and Stieltjes transform
 \begin{eqnarray*}
 s_W(z)=\int_\mathbb{R}\frac{\rho_W(x)}{x-z}dx.
 \end{eqnarray*}
 For $\zeta\geq0$, define the set
 \begin{eqnarray}
 \underline{S}(\zeta):=\{z\in \mathbb{C}: 0\leq E\leq 5\lambda_+, \varphi^{\zeta}p^{-1}\leq\eta\leq 10(1+\frac pn)\}.\label{4.12}
 \end{eqnarray}
 And we say that an event $\Omega$ holds with $\zeta$-high probability if there exists a constant $C>0$ such that
 \begin{eqnarray}
 \mathbb{P}(\Omega^c)\leq p^C\exp(-\varphi^\zeta) \label{4.4}
 \end{eqnarray}
 for large enough $p$. Note that (\ref{4.4}) implies that the event $\Omega$ holds with overwhelming probability if $\zeta>0$. We further denote
 \begin{eqnarray*}
 \Lambda_d:=\max_{k}|G_{kk}-s_W(z)|,~~~~\Lambda_o:=\max_{k\neq l}|G_{kl}|,~~~~\Lambda:=|s_p(z)-s_W(z)|.
 \end{eqnarray*}

 \begin{lemma}\emph{(Theorem 1.5, \cite{PY})} \label{le.4.3} Under the condition $\mathbf{C_1}$, for any $\zeta>0$ there exists a constant $C_\zeta$ such that the following events hold with $\zeta$-high probability.

 (i) The Stieltjes transform of the ESD of $W$ satisfies
 \begin{eqnarray*}
 \bigcap_{z\in\underline{S}(C_\zeta)}\left\{\Lambda(z)\leq \varphi^{C_\zeta}\frac{1}{p\eta}\right\}.
 \end{eqnarray*}

 (ii) The individual matrix elements of the Green function satisfy
 \begin{eqnarray*}
 \bigcap_{z\in \underline{S}(C_\zeta)}\left\{\Lambda_o(z)+\Lambda_d(z)\leq \varphi^{C_\zeta}\left(\sqrt{\frac{\Im s_W(z)}{p\eta}}+\frac{1}{p\eta}\right)\right\}.
 \end{eqnarray*}

 (iii) Uniformly in $E\in\mathbb{R}$,
 \begin{eqnarray*}
 |F_p(E)-F_{W}(E)|\leq \varphi^{C_\zeta}p^{-1}.
 \end{eqnarray*}
 \end{lemma}

 We also need the following lemma on $s_W(z)$.
\begin{lemma}\emph{(Lemma 26, \cite{PY})}
Set $\kappa:=\min(|\lambda_+-E|,|E-\lambda_-|)$. For $z=E+i\eta\in \underline{S}(0)$, (see (\ref{4.12})) we have the following relations:
\begin{eqnarray}
|s_W(z)|\sim 1,~~~~~~~~|1-s_W^2(z)|\sim\sqrt{\kappa+\eta}, \label{4.13}
\end{eqnarray}
\begin{eqnarray}
\Im s_W(z)\sim\left\{
\begin{array}{lll}
\frac{\eta}{\sqrt{\kappa+\eta}} &~~~~~~if~~~\kappa\geq\eta~~~and~~~|E|\not\in[\lambda_-,\lambda_+]\\\\
\sqrt{\kappa+\eta} &~~~~~~if~~~\kappa\leq\eta~~~and~~~|E|\in[\lambda_-,\lambda_+]
\end{array}\right.  \label{4.14}
\end{eqnarray}
where $A\sim B$ means $C^{-1}B\leq A\leq CB$ for some constant $C$. Furthermore
\begin{eqnarray*}
\frac{\Im s_W(z)}{p\eta}\geq O(\frac1p)~~~~ and~~~~\partial_\eta\frac{\Im s_W(z)}{\eta}\leq 0.
\end{eqnarray*}
\end{lemma}

Now we set $Y^{\mathbf{v}}=(\mathbf{y}_{ij}^{\mathbf{v}}):=(x_{ij}^{\mathbf{v}}/||\mathbf{x}_i^{\mathbf{v}}||)_{p, n}$, with elements $x_{ij}^{\mathbf{v}}$ satisfying our basic condition $\mathbf{C_1}$. Correspondingly we let $W^{\mathbf{v}}=Y^{\mathbf{v}}Y^{\mathbf{v}T}$, $G^{\mathbf{v}}(z)=(W^{\mathbf{v}}-z)^{-1}$ and $s_p^{\mathbf{v}}(z)=\frac1pTrG^{\mathbf{v}}(z)$. Define the matrix $W^{\mathbf{w}}$, the Green function $G^{\mathbf{w}}(z)$ and the Stieltjes transform $s_p^{\mathbf{w}}(z)$ analogously for another random sequence $\{x_{ij}^{\mathbf{w}}\}$ satisfying $\mathbf{C}_1$ which is independent of $\{x_{ij}^{\mathbf{v}}\}$. The aim in this section is to prove the following Green function comparison theorem.

Below we only state the results and proofs for the largest eigenvalue. The smallest one is just analogous.
\begin{theorem} \label{th.4.1}\emph{(Green function comparison theorem on the edge)}. Let $F: \mathbb{R}\rightarrow\mathbb{R}$ be a function whose derivatives $F^{(\alpha)}$ satisfy
\begin{eqnarray*}
\max_{x}|F^{(\alpha)}(x)|(|x|+1)^{-C_1}\leq C_1,~~~~\alpha=1,2,3,4
\end{eqnarray*}
with some constant $C_1>0$. Then there exists $\epsilon_0>0$ depending only on $C_1$ such that for any $\epsilon<\epsilon_0$ and for any real numbers $E, E_1$ and $E_2$ satisfying
\begin{eqnarray*}
|E-\lambda_+|\leq p^{-2/3+\epsilon},~~~|E_1-\lambda_+|\leq p^{-2/3+\epsilon},~~~ |E_2-\lambda_+|\leq p^{-2/3+\epsilon},
\end{eqnarray*}
and $\eta=p^{-2/3-\epsilon}$, we have
\begin{eqnarray}
\left|\mathbb{E}^{\mathbf{v}}F(p\eta\Im s_p^{\mathbf{v}}(z))
-\mathbb{E}^{\mathbf{w}}F(p\eta\Im s_p^{\mathbf{w}}(z))\right|\leq Cp^{-1/6+C\epsilon},~~~z=E+i\eta,\label{6.111}
\end{eqnarray}
and
\begin{equation}
\left|\mathbb{E}^{\mathbf{v}}F\big(p\int_{E_1}^{E_2}dx \Im s_p^{\mathbf{v}}(x+i\eta)\big)-\mathbb{E}^{\mathbf{w}}F\big(p\int_{E_1}^{E_2}dx \Im s_p^{\mathbf{w}}(x+i\eta)\big)\right|\leq Cp^{-1/6+C\epsilon}  \label{6.112}
\end{equation}
for some constant $C$ and large enough $p$.
\end{theorem}

\begin{proof}[Proof of Theorem \ref{th.4.1}] The proof is similar to that of Theorem 6.3 of \cite{PY}. Moreover, the proof of (\ref{6.112}) can be taken in a same manner as that of (\ref{6.111}), so we will just present the proof for (\ref{6.111}) below. The basic strategy is to estimate the successive difference of matrices which differ by a row. For $1\leq\gamma\leq p$, we denote by $Y_{\gamma}$ the random matrix whose $j$-th row is the same as that of $Y^{\mathbf{v}}$ if $j\leq\gamma$ and that of $Y^{\mathbf{w}}$ otherwise; in particular $Y_0=Y^{\mathbf{v}}$ and $Y_p=Y^{\mathbf{w}}$. And we set
\begin{eqnarray*}
W_\gamma=Y_\gamma Y^T_\gamma.
\end{eqnarray*}
We shall compare $W_{\gamma-1}$ with $W_\gamma$ by using the following lemma. For simplicity, we denote
\begin{eqnarray*}
s_p^{(i)}(z)=\frac1pTr G^{(i)}(z),~~~\tilde{s}_p^{(i)}(z)=s_p^{(i)}(z)-\frac{1}{pz}.
\end{eqnarray*}
\begin{lemma} \label{le.4.1}
For any sample correlation matrix $W$ with elements satisfying the basic assumption $\mathbf{C_1}$, if $|E-\lambda_+|\leq p^{-2/3+\epsilon}$ and $p^{-2/3}\gg\eta\gg p^{-2/3-\epsilon}$ for some $\epsilon>0$, then we have
\begin{eqnarray*}
\mathbb{E}F(p\eta\Im s_p(z))-\mathbb{E}F(p\eta\Im \tilde{s}_p^{(i)}(z))=A(Y^{(i)},m_1,m_2)+p^{-7/6+C\epsilon}
\end{eqnarray*}
where the functional $A(Y^{(i)}, m_1,m_2)$ only depends on the distribution of $Y^{(i)}$ and the first two moments $m_1, m_2$ of $x_{ij}$.
\end{lemma}
\begin{remark} We always assume $m_1=0$, $m_2=1$ in our case.
\end{remark}
Note that
\begin{eqnarray*}
W_{\gamma-1}^{(\gamma)}=W_{\gamma}^{(\gamma)},
\end{eqnarray*}
thus Lemma \ref{le.4.1} implies that
\begin{eqnarray*}
\mathbb{E}F\left(\eta\Im Tr(W_{\gamma-1}-z)^{-1}\right)-\mathbb{E}F\left(\eta\Im Tr(W_{\gamma}-z)^{-1}\right)=p^{-7/6+C\epsilon}.
\end{eqnarray*}
Then the proof of Theorem \ref{th.4.1} can be completed by the telescoping argument.

Therefore it suffices to prove Lemma \ref{le.4.1} in the sequel. To do this, we need to provide some bounds about $\mathcal{G}^{(i)}$. We only state the result for $i=1$ as the following lemma since the others are analogous.
\begin{lemma} \label{le.4.2} Under the assumptions in Lemma \ref{le.4.1}, we have for $\epsilon>0$ small enough,
\begin{eqnarray}
|\mathbf{y}_1^T(\mathcal{G}^{(1)})^2\mathbf{y}_1|\leq p^{1/3+C\epsilon} \label{4.6}
\end{eqnarray}
and
\begin{eqnarray}
|(\mathcal{G}^{(1)})_{ij}|\leq p^{C\epsilon},~~~|((\mathcal{G}^{(1)})^2)_{ij}|\leq p^{1/3+C\epsilon} \label{4.7}
\end{eqnarray}
hold with overwhelming probability.
\end{lemma}
The proof of Lemma \ref{le.4.2} will be postponed to the end of this section. Now we begin to prove Lemma \ref{le.4.1} assuming Lemma \ref{le.4.2}.
\begin{proof}[Proof of Lemma \ref{le.4.1}] The proof is in a similar manner to that of Lemma 6.5 in \cite{PY}.
At first we rewrite (\ref{2.1}) as
\begin{eqnarray}
G_{11}=\frac{1}{-z-z\mathbf{y}_1^T\mathcal{G}^{(1)}(z)\mathbf{y}_1} \label{4.9}
\end{eqnarray}
by using the facts that
\begin{eqnarray*}
\mathcal{W}^{(1)}\mathcal{G}^{(1)}(z)=I+z\mathcal{G}^{(1)}(z),~~~\mathbf{y}_1^T\mathbf{y}_1=1.
\end{eqnarray*}
Moreover, by Schur's complement, we also have
\begin{eqnarray}
TrG-TrG^{(1)}=G_{11}+\frac{\mathbf{y}_1^TY^{(1)T}(G^{(1)})^2Y^{(1)}\mathbf{y}_1}{-z-z\mathbf{y}_1^T\mathcal{G}^{(1)}(z)\mathbf{y}_1},\label{4.10}
\end{eqnarray}
Inserting (\ref{4.9}) and the identity
\begin{eqnarray*}
Y^{(1)T}(G^{(1)})^2Y^{(1)}=\mathcal{W}^{(1)}(\mathcal{G}^{(1)})^2=\mathcal{G}^{(1)}+z(\mathcal{G}^{(1)})^2
\end{eqnarray*}
into (\ref{4.10}) we can get
\begin{eqnarray}
TrG-TrG^{(1)}+z^{-1}=zG_{11}(\mathbf{y}_1^T(\mathcal{G}^{(1)})^2(z)\mathbf{y}_1).\label{4.15}
\end{eqnarray}

Now we define the quantity $B$ as
\begin{eqnarray*}
B=-zs_W(z)\left[\mathbf{y}_1^T\mathcal{G}^{(1)}(z)\mathbf{y}_1-\left(\frac{-1}{zs_W(z)}-1\right)\right].
\end{eqnarray*}
Thus by (\ref{4.9}) we have
\begin{eqnarray*}
B=-zs_W(z)\left[\left(\frac{-1}{zG_{11}(z)}-1\right)-\left(\frac{-1}{zs_W(z)}-1\right)\right]=\frac{s_W(z)-G_{11}}{G_{11}}.
\end{eqnarray*}
By $(ii)$ of Lemma \ref{le.4.3} and (\ref{4.14}) we can get
\begin{eqnarray*}
|B|\leq p^{-1/3+2\epsilon}\ll 1
\end{eqnarray*}
with overwhelming probability. Thus we have the expansion
\begin{eqnarray}
G_{11}=\frac{s_W(z)}{B+1}=s_W(z)\sum_{k\geq 0}(-B)^{k}.\label{4.16}
\end{eqnarray}

Now we set
\begin{eqnarray*}
y:=\eta(Tr G-TrG^{(1)}+z^{-1}).
\end{eqnarray*}
It follows from (\ref{4.15}) and (\ref{4.16}) that
\begin{eqnarray*}
y=\eta z G_{11}\mathbf{y}_1^T(\mathcal{G}^{(1)})^2\mathbf{y}_1=\sum_{k=1}^{\infty}y_k,
\end{eqnarray*}
where
\begin{eqnarray*}
y_k:=\eta z s_W(z)(-B)^{k-1}\mathbf{y}_1^T(\mathcal{G}^{(1)})^2\mathbf{y}_1.
\end{eqnarray*}

Since $z$ and $s_W(z)$ are $O(1)$ by (\ref{4.13}), by definitions and Lemma \ref{le.4.2}, we have
\begin{eqnarray}
|y_k|\leq O(p^{-k/3+C\epsilon})~~~~\mathrm{and}~~~~|y|\leq O(p^{-1/3+C\epsilon})\label{4.19}
\end{eqnarray}
with overwhelming probability. Thus we have
\begin{eqnarray*}
&&F(p\eta\Im s_p(z))-F(p\eta\Im\tilde{s}_p^{(1)}(z))\nonumber\\
&&=\sum_{k=1}^3\frac{1}{k!}F^{(k)}(p\eta\Im\tilde{s}_p^{(1)}(z))(\Im y)^k+O(p^{-4/3+C\epsilon})
\end{eqnarray*}
with overwhelming probability.

Similarly to the counterpart  proof of Lemma 6.5 in \cite{PY}, we only need to show
\begin{eqnarray}
\quad \quad \mathbb{E}F^{(k)}(p\eta\Im\tilde{s}_p^{(1)}(z))(\Im y)^k=A_k(Y^{(1)},m_1,m_2)+O(p^{-4/3+C\epsilon}),~~~~k=1,2,3 \label{4.18}
\end{eqnarray}
with some functional $A_k$ only depending on the distribution of $Y^{(1)}$, $m_1$ and $m_2$.

Since the proof of (\ref{4.18}) is similar to the counterpart in \cite{PY}, we will only state the proof for $k=3$ below. We use $\mathbb{E}_1$ to denote the expectation with respect to $\mathbf{y}_1$ in the sequel. By using (\ref{4.19}) we obtain
\begin{eqnarray}
F^{(3)}(p\eta\Im\tilde{s}_p^{(1)}(z))(\Im y)^3=F^{(3)}(p\eta\Im\tilde{s}_p^{(1)}(z))(\Im y_1)^3+O(p^{-4/3+C\epsilon}) \label{4.22}
\end{eqnarray}
with overwhelming probability.
If we write $r_1=\Re(\eta z s_W(z)), r_2=\Im(\eta z s_W(z))$, then we have
\begin{eqnarray}
\mathbb{E}_1(\Im y_1)^3 &=&\mathbb{E}_1r_1^3(\Im(\mathbf{y}_1^T(\mathcal{G}^{(1)})^2\mathbf{y}_1))^3+\mathbb{E}_1r_2^3(\Re(\mathbf{y}_1^T(\mathcal{G}^{(1)})^2\mathbf{y}_1))^3\nonumber\\
&&+3\mathbb{E}_1r_1^2r_2(\Im(\mathbf{y}_1^T(\mathcal{G}^{(1)})^2\mathbf{y}_1))^2(\Re(\mathbf{y}_1^T(\mathcal{G}^{(1)})^2\mathbf{y}_1))\nonumber\\
&&+3\mathbb{E}_1r_1r_2^2(\Im(\mathbf{y}_1^T(\mathcal{G}^{(1)})^2\mathbf{y}_1))(\Re(\mathbf{y}_1^T(\mathcal{G}^{(1)})^2\mathbf{y}_1))^2\nonumber\\
&=&r_1^3\sum_{k_1,\cdots,k_6}\mathbb{E}_1(\prod_{i=1}^6\frac{x_{1k_i}}
{||\mathbf{x}_1||})\prod_{i=1}^{3}\Im\big((\mathcal{G}^{(1)})^2\big)_{k_{2i-1},k_{2i}}\nonumber\\
&&+r_2^3\sum_{k_1,\cdots,k_6}\mathbb{E}_1(\prod_{i=1}^6\frac{x_{1k_i}}
{||\mathbf{x}_1||})\prod_{i=1}^{3}\Re\big((\mathcal{G}^{(1)})^2\big)_{k_{2i-1},k_{2i}}\nonumber\\
&&+3r_1r_2^2\mathbb{E}_1(\prod_{i=1}^6\frac{x_{1k_i}}
{||\mathbf{x}_1||})\prod_{i=1}^{2}\Re\big((\mathcal{G}^{(1)})^2\big)_{k_{2i-1},k_{2i}}\Im\big((\mathcal{G}^{(1)})^2\big)_{k_5,k_6}\nonumber\\
&&+3r_1^2r_2\mathbb{E}_1(\prod_{i=1}^6\frac{x_{1k_i}}
{||\mathbf{x}_1||})\prod_{i=1}^{2}\Im\big((\mathcal{G}^{(1)})^2\big)_{k_{2i-1},k_{2i}}\Re\big((\mathcal{G}^{(1)})^2\big)_{k_5,k_6}. \label{4.17}
\end{eqnarray}

 Notice that if there exists a $k_i$ which appears only once in the above product, then by the assumption that $x_{ij}$ is symmetric, we have
 \begin{eqnarray}
 \mathbb{E}_1(\prod_{i=1}^6\frac{x_{1k_i}}{||\mathbf{x}_1||})=0=m_1.\label{4.20}
 \end{eqnarray}
 So we consider the case where $k_i$ appears exactly twice.
Firstly, we consider
 \begin{eqnarray*}
||\mathbf{x}_1||^6=\sum_{k_1,k_2,k_3}x_{1k_1}^2x_{1k_2}^2x_{1k_3}^2:=\sum_{(1)}x_{1k_1}^2x_{1k_2}^2x_{1k_3}^2+\sum_{(2)}x_{1k_1}^2x_{1k_2}^2x_{1k_3}^2,
 \end{eqnarray*}
 where the first summation goes through the indices $k_1,k_2,k_3$ such that they are not equal to each other, and the second summation goes through the left part of the indices. Then it is not difficult to see the number of the terms in the second summation is of the order $O(n^2)$. By the exponential tail assumption and the Hoeffding inequality, we can see
 \begin{eqnarray*}
 \mathbb{E}_1\sum_{(2)}\frac{x_{1k_1}^2x_{1k_2}^2x_{1k_3}^2}{||\mathbf{x}_1||^6}=O(\frac{\log^{O(1)}n}{n}).
 \end{eqnarray*}
 Furthermore, since $x_{11},\cdots,x_{1n}$ are i.i.d., we have for $k_1,k_2,k_3$ not equal to each other
 \begin{equation}
\quad \quad \mathbb{E}_1\frac{x_{1k_1}^2x_{1k_2}^2x_{1k_3}^2}{||\mathbf{x}_1||^6}
 =\frac{1}{n(n-1)(n-2)}(1-O(\frac{\log^{O(1)}n}{n}))=\frac{m_2^3}{n^3}+O(\frac{\log^{O(1)}n}{n^4}). \label{4.21}
 \end{equation}
 Therefore by (\ref{4.17}), (\ref{4.20}), (\ref{4.21}) and the fact that $\mathcal{G}^{(1)}$ only depends on $Y^{(1)}$, we have
 \begin{eqnarray}
 &&|\mathbb{E}_1(\Im y_1)^3-\tilde{A}_3(Y^{(1)},m_1, m_2)|\nonumber\\
 &\leq&\frac{\log^{O(1)}n}{n^4}|\eta z s_W(z)|^3\sum_{(3)}|[(\mathcal{G}^{(1)})^2]_{k_{1},k_{2}}[(\mathcal{G}^{(1)})^2]_{k_{3},k_{4}}[(\mathcal{G}^{(1)})^2]_{k_{5},k_{6}}|\nonumber\\
 &&+C|\eta z s_W(z)|^3\sum_{(4),(5)}\mathbb{E}_1\bigg{|}\prod_{i=1}^6\frac{x_{1k_i}}{||\mathbf{x}_1||}\bigg{|}\cdot
 |[(\mathcal{G}^{(1)})^2]_{k_{1},k_{2}}[(\mathcal{G}^{(1)})^2]_{k_{3},k_{4}}[(\mathcal{G}^{(1)})^2]_{k_{5},k_{6}}|
 \label{4.5}
 \end{eqnarray}
  with some functional $\tilde{A}_3$ only depending on the distribution of $Y^{(1)}$, $m_1$ and $m_2$. Here the first summation $\sum_{(3)}$ in (\ref{4.5}) goes through the terms such that each $k_i,i=1,\cdots,6$ appears exactly twice. It is easy to see that there are $O(n^3)$ such terms totally. And the second summation goes through the terms such that $(4)$ no $k_i$ appears only once and $(5)$ at least one $k_i$ appears three times. Thus we have the total number of the terms in the second summation is of the order $O(n^2)$. Then by using Lemma \ref{le.4.2} and the fact
  \begin{eqnarray*}
  \mathbb{E}_1\bigg{|}\prod_{i=1}^6\frac{x_{1k_i}}{||\mathbf{x}_1||}\bigg{|}=O(\frac{\log^{O(1)}n}{n^3}),
  \end{eqnarray*}
  we have
 \begin{eqnarray}
 \mathbb{E}_1(\Im y_1)^3=\tilde{A}_3(Y^{(1)},m_1,m_2)+O(p^{-2+C\epsilon}) \label{4.23}
 \end{eqnarray}
 By inserting (\ref{4.23}) into (\ref{4.22}), we can get (\ref{4.18}) for $k=3$. The cases of $k=1$ and $k=2$ can be proved similarly by inserting Lemma \ref{le.4.2}. So we conclude the proof.
\end{proof}

Now we begin to prove Lemma \ref{le.4.2}.
\begin{proof}[Proof of Lemma \ref{le.4.2}]
The proof of (\ref{4.6}) is the same as the counterpart in \cite{PY}, (see (6.36) of \cite{PY}). So we only state the proof of (\ref{4.7}) below. For the ease of the presentation, we  prove (\ref{4.7}) for $\mathcal{G}=(\mathcal{W}-z)^{-1}:=(Y^TY-z)^{-1}$ instead of $\mathcal{G}^{(1)}$. By the spectral decomposition, we have
\begin{eqnarray*}
\mathcal{G}^\alpha=\sum_{k=1}^p\frac{1}{(\lambda_k-z)^\alpha}u_ku_k^T+\frac{1}{(-z)^\alpha}P,~~~~~\alpha=1,2,
\end{eqnarray*}
where the projection $P=I-\sum_{k=1}^pu_ku_k^T$. Consequently, we have
\begin{eqnarray*}
(\mathcal{G}^\alpha)_{ij}=\sum_{k=1}^p\frac{1}{(\lambda_k-z)^\alpha}u_{ki}u_{kj}+\frac{1}{(-z)^\alpha}P_{ij}.
\end{eqnarray*}
Note that $|P_{ij}|\leq 1,|z|\geq \lambda_+/2$. By the delocalization property of $u_k$ in Theorem \ref{th.2.2} one has
\begin{eqnarray*}
|(\mathcal{G}^\alpha)_{ij}|\leq\frac{\log^{O(1)}p}{p}\sum_{k=1}^p\frac{1}{|\lambda_k-z|^\alpha}+C
\end{eqnarray*}
 with overwhelming probability. For $\alpha=2$, by using $i)$ of Lemma \ref{le.4.3} and (\ref{4.14}) we have
\begin{eqnarray*}
\sum_{k=1}^p\frac{1}{|\lambda_k-z|^2}=p\eta^{-1}\Im s_p(z)\leq p^{\epsilon}p\eta^{-1}\frac{1}{p\eta}\leq p^{4/3+C\epsilon},
\end{eqnarray*}
which implies
\begin{eqnarray*}
|(\mathcal{G}^2)_{ij}|\leq p^{1/3+C\epsilon}.
\end{eqnarray*}
For $\alpha=1$, we have
\begin{eqnarray*}
\sum_{k=1}^p\frac{1}{|\lambda_k-z|}=p\int\frac{1}{|x-z|}dF_p(x).
\end{eqnarray*}
Observe that
\begin{eqnarray*}
|p\int\frac{1}{|x-z|}dF_p(x)-p\int\frac{1}{|x-z|}dF_{W}(x)|\leq Cp \int\frac{|F_p(x)-F_{W}(x)|}{|x-z|^2}dx\leq \eta^{-1}p^{C\epsilon}
\end{eqnarray*}
with overwhelming probability. Here we used $(iii)$ of Lemma \ref{le.4.3} in the last inequality. Consequently, we have
\begin{eqnarray*}
|\mathcal{G}_{ij}|\leq (\log^{O(1)}p)\int\frac{1}{|x-z|}dF_{W}(x)+C.
\end{eqnarray*}
It remains to estimate $\int\frac{1}{|x-z|}dF_{W}(x)$.  For $E<\lambda_+$ such that $\lambda_+-E\leq p^{-2/3+\epsilon}$
\begin{eqnarray*}
\int\frac{1}{|x-z|}dF_{W}(x)=\left(\int_{\lambda_-}^{2E-\lambda_+}+\int_{2E-\lambda_+}^{\lambda_+}\right)\frac{1}{\sqrt{(x-E)^2+\eta^2}}dF_{W}(x).
\end{eqnarray*}
By the formula for the MP law, one has
\begin{eqnarray}
\int_{\lambda_-}^{2E-\lambda_+}\frac{1}{\sqrt{(x-E)^2+\eta^2}}dF_{W}(x)&\leq& C\int_{\lambda_-}^{2E-\lambda_+}\frac{\sqrt{\lambda_+-x}}{E-x}dx \nonumber\\
&\leq& C\int_{\lambda_-}^{2E-\lambda_+}\frac{1}{\sqrt{E-x}}dx=O(1),
\label{4.8}
\end{eqnarray}
and
\begin{eqnarray*}
\int_{2E-\lambda_+}^{\lambda_+}\frac{1}{\sqrt{(x-E)^2+\eta^2}}dF_{W}(x)\leq \eta^{-1}\int_{2E-\lambda_+}^{\lambda_+}dF_{W}(x)=o(1).
\end{eqnarray*}
When $E\geq \lambda_+$, we still have (\ref{4.8}). Therefore, we have
\begin{eqnarray*}
|\mathcal{G}_{ij}|\leq p^{C\epsilon}
\end{eqnarray*}
with overwhelming probability. Thus we complete the proof.
\end{proof}
Theorem \ref{th.4.1} is proved.
\end{proof}

\section{Proofs of main theorems}
In this section, we provide the proofs of Theorem \ref{th.1.2} and Theorem \ref{th.1.4}.
\begin{proof} [Proof of Theorem \ref{th.1.2}]
The proof of Theorem \ref{th.1.2} is totally based on Theorem 1.5 of \cite{PY} and our Theorem \ref{th.4.1}. Let $W^{\mathbf{v}}$ and $W^{\mathbf{w}}$ be two independent sample correlation matrix satisfying $\mathbf{C}_1$. We claim that there is an $\varepsilon>0$ and $\delta>0$ such that for any real number $s$ (which may depend on $p$) one has
\begin{eqnarray}
\mathbb{P}^{\mathbf{v}}(p^{2/3}(\lambda_p-\lambda_+)\leq s-p^{-\varepsilon})-p^{-\delta}
\leq\mathbb{P}^{\mathbf{w}}(p^{2/3}(\lambda_p-\lambda_+)\leq s)\nonumber\\
\leq
\mathbb{P}^{\mathbf{v}}(p^{2/3}(\lambda_p-\lambda_+)\leq s+p^{-\varepsilon})+p^{-\delta}
\label{5.555}
\end{eqnarray}
for $p\geq p_0$ sufficiently large, where $p_0$ is independent of $s$.
The proof of (\ref{5.555}) is independent of the matrix model and totally based on Theorem 1.5 of \cite{PY} and our Theorem \ref{th.4.1}, we refer to the proof of Theorem 1.7 of \cite{PY} for details.

Now if we choose $W^{\mathbf{v}}$ to be the Bernoulli case, it is not difficult to get Theorem \ref{th.1.2} by combining (\ref{5.555}) and Theorem \ref{th.0.1}.
\end{proof}
\begin{proof}[Proof of Theorem \ref{th.1.4}]
Set the matrix
\begin{eqnarray*}
A=\left(
\begin{array}{cccccc}
\frac{1}{\sqrt{n}} &\frac{1}{\sqrt{n}} &\frac{1}{\sqrt{n}} &\cdots &\frac{1}{\sqrt{n}}\\
\frac{1}{\sqrt{2}} &-\frac{1}{\sqrt{2}} &0 &\cdots &0\\
\frac{1}{\sqrt{3\cdot2}} &\frac{1}{\sqrt{3\cdot 2}} &-\frac{2}{\sqrt{3\cdot2}} &\cdots &0\\
\cdots &\cdots &\cdots &\cdots &\cdots\\
\frac{1}{\sqrt{n(n-1)}} &\frac{1}{\sqrt{n(n-1)}} &\frac{1}{\sqrt{n(n-1)}} &\cdots &-\frac{n-1}{\sqrt{n(n-1)}}
\end{array}
\right).
\end{eqnarray*}
It is easy to see $A$ is an orthogonal matrix. Moreover, it is elementary that
\begin{eqnarray*}
A(x_{i1}-\bar{x}_i,\cdots, x_{in}-\bar{x}_i)^T=(0,z_{i1},\cdots,z_{i,n-1})^T,
\end{eqnarray*}
 where $z_{i1},\cdots, z_{in-1}$ is a sequence of i.i.d $N(0,1)$ variables. Further, if we denote the vector $\mathbf{z}_i=(z_{i1},\cdots, z_{in-1})^T$, we also have
\begin{eqnarray*}
||\mathbf{x}_i-\bar{x}_i||^2=\sum_{k=1}^{n-1}z_{ik}^2=||\mathbf{z}_i||^2.
\end{eqnarray*}
Thus one has
\begin{eqnarray*}
\mathcal{R}=RR^T=RA^TAR^T=:\mathcal{Z}.
\end{eqnarray*}
Here
\begin{eqnarray*}
\mathcal{Z}=ZZ^T
\end{eqnarray*}
with
\begin{eqnarray*}
Z=\left(
\begin{array}{cccc}
\frac{z_{11}}{||\mathbf{z}_{1}||} &\frac{z_{12}}{||\mathbf{z}_{1}||} &\cdots &\frac{z_{1,n-1}}{||\mathbf{z}_{1}||}\\
\vdots &\vdots &\vdots &\vdots\\
\frac{z_{p1}}{||\mathbf{z}_p||} &\frac{z_{p2}}{||\mathbf{z}_p||} &\cdots &\frac{z_{p,n-1}}{||\mathbf{z}_p||}
\end{array}
\right).
\end{eqnarray*}
Consequently, in the Gaussian case, $\mathcal{R}$ is also a $W$-type sample correlation matrix defined in (\ref{1.1}) with parameters $p,n-1$. Thus by Theorem \ref{th.1.2}, we have
\begin{eqnarray}
\frac{(n-1)\lambda_p(\mathcal{R})-(p^{1/2}+(n-1)^{1/2})^2}{((n-1)^{1/2}+p^{1/2})(p^{-1/2}+(n-1)^{-1/2})^{1/3}}\stackrel d\longrightarrow TW_1 \label{6.113}
\end{eqnarray}
and
\begin{eqnarray}
\frac{(n-1)\lambda_1(\mathcal{R})-(p^{1/2}-(n-1)^{1/2})^2}{((n-1)^{1/2}-p^{1/2})(p^{-1/2}-(n-1)^{-1/2})^{1/3}}\stackrel d\longrightarrow TW_1.\label{6.114}
\end{eqnarray}
as $p\rightarrow\infty$. Replacing $n-1$ by $n$ in (\ref{6.113}) and (\ref{6.114}), we can complete the proof of Theorem \ref{th.1.4}.
\end{proof}
\section{Appendix A}
In this appendix we prove Lemma \ref{le.rrr}
\begin{proof}[Proof of Lemma \ref{le.rrr}]
At first we prove $i)$. Note that $W=DSD$.
For $W$ and $SD^2$ share the same eigenvalues, it is equivalent to prove that the eigenvalues of $SD^2$ are simple. We further introduce the polynomial $P_1(X)$ of $\{x_{ij}, 1\leq i\leq p,1\leq j\leq n\}$ as
\begin{eqnarray*}
P_1(X)=\prod_{k=1}^p||\mathbf{x}_k||^2.
\end{eqnarray*}
It is easy to see $P_1(X)$ vanishes with zero Lebesgue measure, so we can always assume $P_1(X)\neq 0$. As a consequence, we can reduce our problem to prove the matrix
\begin{eqnarray*}
Q:=SD^2P_1(X)
\end{eqnarray*}
has no multiple eigenvalue. Now we denote the discriminant of the characteristic polynomial of $Q$ by $P_Q(X)$. Observe that all the entries of $Q$ are polynomials of $\{x_{ij}, 1\leq i\leq p,1\leq j\leq n\}$, so $P_Q(X)$ is also a polynomial of $\{x_{ij}, 1\leq i\leq p,1\leq j\leq n\}$. For the set of zeros of any non null polynomial in real variables only has zero Lebesgue measure, it suffices to prove that $P_Q(X)$ is not a null polynomial. In other words, it suffices to find a family $\{x_{ij}, 1\leq i\leq p,1\leq j\leq n\}$ such that $P_Q(X)\neq 0$. It is equivalent to show that $W$ has no multiple eigenvalue for one sample of the collection $\{x_{ij}, 1\leq i\leq p,1\leq j\leq n\}$ such that $P_1(X)\neq 0$.

Now we choose the sample as
\begin{eqnarray*}
x_{ij}=\left\{
\begin{array}{ccc}
1, &\mbox{$j=i$~or~$i+1$}\\
0, &\mbox{others}
\end{array}
\right.
\end{eqnarray*}
with $1\leq i\leq p,1\leq j\leq n$. Then it is not difficult to see
\begin{eqnarray*}
W=\left(
\begin{array}{ccccc}
1 &\frac12 &~ &~ &~\\
\frac12 &1 &\frac12 &~ &~\\
~ &\ddots &\ddots  &\ddots &~\\
~ &~ &\frac12 &1 &\frac12\\
~ &~ &~ &\frac12 &1
\end{array}
\right),
\end{eqnarray*}
which is a Jacobi matrix with positive subdiagonal entries. Such a Jacobi matrix has simple eigenvalues, for example, see Proposition 2.40 of \cite{Deift}.

Next we turn to the proof of $ii)$. We use $X^{(p)}$ to denote the submatrix of $X$ with $p$-th row deleted, and use $D^{(p)}$ to denote the $p-1\times p-1$ upper left corner of $D$. And we set $S^{(p)}=X^{(p)}X^{(p)T}$, thus one has $W^{(p)}=D^{(p)}S^{(p)}D^{(p)}$. Similar to the proof of $i)$, we can prove that $SD^2P_1(X)$ and $S^{(p)}(D^{(p)})^2P_1(X)$ have no eigenvalue in common instead. It is easy to see the resultant of the characteristic polynomials of $SD^2P_1(X)$ and $S^{(p)}(D^{(p)})^2P_1(X)$ is a polynomial of $\{x_{ij}, 1\leq i\leq p,1\leq j\leq n\}$. Therefore, it suffices to show the resultant is a non null polynomial. Equivalently, we shall provide a sample of $\{x_{ij}, 1\leq i\leq p,1\leq j\leq n\}$ such that $W$ and $W^{(p)}$ have no eigenvalue in common.

Using $i)$ to $W^{(p)}$ we can denote the ordered eigenvalues of $W^{(p)}$ by $\lambda_1^{(p)}<\lambda_2^{(p)}<\cdots<\lambda_{p-1}^{(p)}$. By Cauchy's interlacing property, one has
\begin{eqnarray}
0\leq\lambda_1\leq \lambda_1^{(p)}\leq\lambda_2\leq\cdots\leq\lambda_{p-1}^{(n)}\leq\lambda_p. \label{3.332}
\end{eqnarray}
Moreover, we know that $\mathcal{W}^{(p)}$ shares the same nonzero eigenvalues with $W^{(p)}$. So we can provide an example such that $\mathcal{W}$ and $\mathcal{W}^{(p)}$ have no nonzero eigenvalue in common instead. Note
\begin{eqnarray}
\mathcal{W}=\mathcal{W}^{(p)}+\mathbf{y}_p\mathbf{y}_p^T.\label{3.333}
\end{eqnarray}
Taking trace on both side of (\ref{3.333}), we obtain
\begin{eqnarray}
\lambda_1+\cdots+\lambda_p=\lambda_1^{(n)}+\cdots+\lambda_{p-1}^{(n)}+1.\label{3.334}
\end{eqnarray}
Now if we fix $\{x_{ij},1\leq i\leq p-1,1\leq j\leq n\}$ such that $\lambda_1^{(p)}<\lambda_2^{(p)}<\cdots<\lambda_{p-1}^{(p)}$ and let $\{x_{pj},1\leq j\leq n\}$ vary. When $\{x_{pj},1\leq j\leq n\}$ runs through the set $\mathbb{R}^{n}$, the ordered nonzero eigenvalues of $\mathcal{W}$ describe the set of families $\lambda_1,\cdots,\lambda_p$ of real numbers obeying (\ref{3.332}) and (\ref{3.334}), see the proof of Lemma 11.4 of \cite{BGM} for example. Thus it is easy to find a family $\lambda_1,\cdots,\lambda_p$ such that
\begin{eqnarray*}
\{\lambda_1,\cdots,\lambda_p\}\cap\{\lambda_1^{(p)},\lambda_2^{(p)},\lambda_{p-1}^{(p)}\}=\emptyset
\end{eqnarray*}
\\

Now we prove $iii)$. We set $X_{(n)}$ to be the submatrix of $X$ with the $n$-th column deleted
and set
\begin{eqnarray*}
\widehat{D}_{(n)}=\left(
\begin{array}{cccc}
\frac{\sqrt{n}}{||\widehat{\mathbf{x}}_1||} &~ &~\\
~ &\ddots &~\\
~ &~ &\frac{\sqrt{n}}{||\widehat{\mathbf{x}}_p||}
\end{array}
\right).
\end{eqnarray*}
Let $S_{(n)}=X_{(n)}X_{(n)}^T$. It is obvious that $S_{(n)}\widehat{D}_{(n)}^2$ shares the same eigenvalues with $\widehat{W}_{(n)}$
Now we introduce the polynomials
\begin{eqnarray*}
P_2(X)=\prod_{k=1}^p||\mathbf{x}_k||^2\cdot||\widehat{\mathbf{x}}_k||^2.
\end{eqnarray*}
To prove that $W$ and $\widehat{W}_{(n)}$ have no eigenvalue in common, we only need to show $SD^2$ and $S_{(n)}\widehat{D}_{(n)}^2$ have no eigenvalue in common. Moreover, if $P_2(X)$ does not vanish, it is equivalent to prove that the matrices $T:=SD^2P_2(X)$ and $\widehat{T}_{(n)}:=S_{(n)}\widehat{D}_{(n)}^2P_2(X)$ have no eigenvalue in common. Note that the event $P_2(X)=0$ has zero Lebesgue measure. What's more, it is not difficult to see the entries of $T$ and $\widehat{T}_{(n)}$ are all polynomials of the elements of $X$, thus the resultant $R(X)$ of the characteristic polynomials of $T$ and $T_{(n)}$ is also a polynomial of the elements of $X$.  Therefore, we only need to show $R(X)$ is a non null polynomial, it suffices to give only one example of $X$ such that $W$ and $\widehat{W}_{(n)}$ do not have eigenvalue in common. For example, we can choose
\begin{eqnarray*}
x_{ij}=
\left\{
\begin{array}{lll}
1,~~~~\mbox{$j=i$ or $j=n$},\\
0, ~~~~\mbox{others}
\end{array}
\right.
\end{eqnarray*}
Then we have $\widehat{W}_{(n)}=I_p$ and
\begin{eqnarray*}
W=\left(
\begin{array}{ccccc}
1 &\frac12  &\cdots &\frac12\\
\frac12 &1  &\cdots &\frac12\\
\vdots &\vdots &\ddots &\vdots\\
\frac12 &\frac12  &\cdots &1
\end{array}
\right).
\end{eqnarray*}
Thus it is easy to see $\widehat{W}_{(n)}$ and $W$ have no eigenvalue in common for $\det(W-I)\neq 0$, which implies that $R(X)$ is not a null polynomial, so we conclude the proof.
\end{proof}
\section{Appendix B}
In this appendix, we prove Lemma \ref{le.3.5}.
 If we denote
 \begin{eqnarray*}
 \hat{h}_n=(\frac{x_{1n}}{||\widehat{\mathbf{x}}_1||},\frac{x_{2n}}{||\widehat{\mathbf{x}}_2||},\cdots,\frac{x_{pn}}{||\widehat{\mathbf{x}}_p||})^T.
 \end{eqnarray*}
Set
\begin{eqnarray*}
c_i=\frac{1}{||\mathbf{x}_i||\cdot||\widehat{\mathbf{x}}_i||\cdot(||\mathbf{x}_i||+||\widehat{\mathbf{x}}_i||)}.
\end{eqnarray*}
By the Hoeffding inequality, we have
\begin{eqnarray}
c_i=\frac{1}{n^{3/2}}+\frac{K^{O(1)}\log^{O(1)} n}{n^2}.\label{6.001}
\end{eqnarray}
holds with overwhelming probability.
It is not difficult to see
 \begin{eqnarray}
\hat{v}_j\cdot h_n=\hat{v}_{j}\cdot\hat{h}_n- \hat{v}_j\cdot (c_1x_{1n}^3,\cdots,c_px_{pn}^3)^T:=\hat{v}_{j}\cdot\hat{h}_n+d_j.\label{3.47}
 \end{eqnarray}
 By (\ref{6.001}), we can write $d_j$
 \begin{eqnarray}
 d_{j}:=\frac{1}{n^{3/2}}\hat{v}_j\cdot (x_{1n}^3,\cdots,x_{pn}^3)^T+f_j. \label{6.003}
 \end{eqnarray}
  Observe that
  \begin{eqnarray*}
  \sum_{j\in J}(\hat{v}_{j}\cdot h_n)^{2}=\sum_{j\in J}(\hat{v}_{j}\cdot \hat{h}_n)^{2}+2\sum_{j\in J}d_j(\hat{v}_{j}\cdot \hat{h}_n)+\sum_{j\in J}d_j^2.
  \end{eqnarray*}
  Since $(x_{1n}^3,\cdots,x_{pn}^3)^T$ is also a random vector with mean zero and finite variance entries, Lemma \ref{le.2.5} can be used to the first part of the right hand side of (\ref{6.003}). Thus if we set the projection
  \begin{eqnarray}P_J=\sum_{j\in J}\hat{v}_j\hat{v}_j^T, \label{6.000}
  \end{eqnarray}
 then we have
  \begin{eqnarray*}
  \sum_{j\in J}d_j^2\leq C\frac{1}{n^3}|P_J \cdot (x_{1n}^3,\cdots,x_{pn}^3)^T|^2+C\sum_{j\in J}f_j^2=O(\frac{|J|}{n^3})+O(\frac{K^{O(1)}\log^{O(1)}n}{n^3}).
  \end{eqnarray*}
  with overwhelming probability. Here we have used the fact that for any $J$
  \begin{eqnarray*}
  \sum_{j\in J}f_j^2\leq C\frac{K^{O(1)}\log^{O(1)}n}{n^4}\sum_{i=1}^px_{in}^6=O(\frac{K^{O(1)}\log^{O(1)}n}{n^3})
  \end{eqnarray*}
  with overwhelming probability.
   Since
  \begin{eqnarray*}
  \sum_{j\in J}d_j(\hat{v}_{j}\cdot \hat{h}_n)\leq\bigg{(}\sum_{j\in J}d_j^{2}\bigg{)}^{1/2}\bigg{(}\sum_{j\in J}(\hat{v}_{j}\cdot \hat{h}_n)^2\bigg{)}^{1/2},
  \end{eqnarray*}
  it suffices to prove the following lemma instead.
 \begin{lemma} \label{le.5.1} Using the notation in Lemma \ref{le.3.5}, we have for any $J\in \{1,\cdots,p\}$ with $|J|=d\leq nK^{-3}$,
\begin{eqnarray*}
\sqrt{n}[\sum_{j\in J}(\hat{v}_{j}\cdot \hat{h}_n)^{2}]^{1/2}=\sqrt{d}+O(K\log n)
\end{eqnarray*}
with overwhelming probability.
\end{lemma}
 \begin{proof}
Observe that
 \begin{eqnarray*}
 \hat{v}_{j}\cdot\hat{h}_n=(\frac{\hat{v}_{j1}}{||\widehat{\mathbf{x}}_1||},\frac{\hat{v}_{j2}}{||\widehat{\mathbf{x}}_2||},
 \cdots,\frac{\hat{v}_{jp}}{||\widehat{\mathbf{x}}_p||})^T\cdot(x_{1n},\cdots,x_{pn})^T .
 \end{eqnarray*}
 Now we set
 \begin{eqnarray*}
 \tilde{v}_j=\sqrt{n}(\frac{\hat{v}_{j1}}{||\widehat{\mathbf{x}}_1||},\frac{\hat{v}_{j2}}{||\widehat{\mathbf{x}}_2||},
 \cdots,\frac{\hat{v}_{jp}}{||\widehat{\mathbf{x}}_p||})^T
 \end{eqnarray*}
 and
 \begin{eqnarray*}
 \tilde{h}_n=(x_{1n},\cdots,x_{pn})^T.
 \end{eqnarray*}
 It follows that
 \begin{eqnarray*}
 \hat{v}_{j}\cdot\hat{h}_n=\frac{1}{\sqrt{n}}\tilde{v}_j\cdot\tilde{h}_n.
 \end{eqnarray*}
We use the following concentration theorem, which is a consequence of Talagrand's inequality, (see Theorem $69$ of \cite{TV1}).
\begin{theorem} \label{th.5.1}(\emph{Talagrand's inequality}). Let $\mathbf{D}$ be the disk $\{z\in\mathbb{C}, |z|\leq K\}$. For every product probability $\mu$ on $\mathbf{D}^p$, every convex 1-Lipschitz function $F:\mathbb{C}^p\rightarrow \mathbb{R}$, and every $r\geq 0$,
\begin{eqnarray*}
\mu(|F-M(F)|\geq r)\leq 4\exp(-r^2/16K^2),
\end{eqnarray*}
where $M(F)$ denotes the median of $F$.
\end{theorem}
\begin{remark}
In fact, here we only need the real case of the theorem.
\end{remark}
It is easy to see
\begin{eqnarray*}
\sqrt{n}[\sum_{j\in J}(\hat{v}_j\cdot\hat{h}_n)^2]^{1/2}=[\sum_{j\in J}(\tilde{v}_j\cdot \tilde{h}_n)^2]^{1/2}=:F(\tilde{h}_n)
\end{eqnarray*}
is a convex function of the vector $\tilde{h}_n$. Note
\begin{eqnarray*}
\frac{|F(\tilde{h}'_n)-F(\tilde{h}_n)|}{||\tilde{h}'_n-\tilde{h}_n||}
=\frac{|F(\tilde{h}'_n)-F(\tilde{h}_n)|}{||\sqrt{n}\hat{h}'_n-\sqrt{n}\hat{h}_n||}\cdot\frac{||\sqrt{n}\hat{h}'_n-\sqrt{n}\hat{h}_n||}{||\tilde{h}'_n-\tilde{h}_n||},
\end{eqnarray*}
where
\begin{eqnarray*}
\hat{h}'_n=(\frac{x'_{1n}}{||\widehat{\mathbf{x}}_1||},\cdots,\frac{x'_{pn}}{||\widehat{\mathbf{x}}_p||})^T, \quad \tilde{h}'_n=(x'_{1n},\cdots,x'_{pn})^T.
\end{eqnarray*}
Since $F(\tilde{h}_n)$ is the norm of a projection of the vector $\sqrt{n}\hat{h}_n$,  it is always 1-Lipschitz with respect to $\sqrt{n}\hat{h}_n$. And by the Hoeffding inequality, we also have
\begin{eqnarray*}
\frac{||\sqrt{n}\hat{h}'_n-\sqrt{n}\hat{h}_n||}{||\tilde{h}'_n-\tilde{h}_n||}\leq 2
\end{eqnarray*}
with overwhelming probability. So $F(\tilde{h}_n)$ is a 2-Lipschitz function with overwhelming probability. Thus we can always consider $F(\tilde{h}_n)$ as a 2-Lipschitz function below. By Theorem \ref{th.5.1}, we have
\begin{eqnarray*}
\mathbb{P}(|F(\tilde{h}_n)-M(F(\tilde{h}_n))|\geq r)\leq 4\exp(-r^2/64K^2).
\end{eqnarray*}
So to conclude the proof of Lemma \ref{le.5.1}, we only need to show that
\begin{eqnarray*}
|M(F(\tilde{h}_n))-\sqrt{d}|\leq 2K.
\end{eqnarray*}
Note that
\begin{eqnarray*}
\sqrt{n}\hat{h}_n=\left(
\begin{array}{ccc}
\frac{\sqrt{n}}{||\widehat{\mathbf{x}}_1||} &~ &~\\
~ &\ddots &~\\
~ &~ &\frac{\sqrt{n}}{||\widehat{\mathbf{x}}_p||}
\end{array}
\right)
\tilde{h}_n:=\hat{D}\tilde{h}_n
\end{eqnarray*}
So we have
$
F^2(\tilde{h}_n)=n\sum_{j\in J}|\hat{v}_j\cdot \hat{h}_n|^2=\sum_{j\in J}\tilde{h}^T_n \hat{D}\hat{v}_j\hat{v}_j^T\hat{D}\tilde{h}_n:=\tilde{h}_n^T \hat{D}P_J\hat{D}\tilde{h}_n,
$
where $P_J$ is the projection defined in (\ref{6.000}).
Let $\hat{D}P_J\hat{D}=:(m_{kl})_{1\leq k,l\leq p}$, then we have
\begin{eqnarray*}
F^2(\tilde{h}_n)=\sum_{1\leq k,l\leq p} m_{kl}x_{kn}x_{ln}=\sum_{k=1}^p m_{kk}x_{kn}^2+\sum_{1\leq k\neq l\leq p}m_{kl}x_{kn}x_{ln}.
\end{eqnarray*}
We fix all the variables except $x_{1n},\cdots, x_{pn}$, so the probabilities and expectations are all taken with respect to $\tilde{h}_n$ below. Consider the event $\mathcal{E}_{+}$ that $F(\tilde{h}_n)\geq \sqrt{d}+2K$, which implies $F^2(\tilde{h}_n)\geq d+4\sqrt{d}K+K^2$. It follows that
\begin{eqnarray*}
\mathbb{P}(\mathcal{E}_{+})\leq \mathbb{P}(\sum_{k=1}^pm_{kk}x_{kn}^2\geq d+2\sqrt{d}K)+\mathbb{P}(|\sum_{1\leq k\neq l\leq p}m_{kl}x_{kn}x_{ln}|\geq 2\sqrt{d}K).
\end{eqnarray*}
Observe that
\begin{eqnarray*}
\mathbb{E}\{\sum_{k=1}^pm_{kk}x_{kn}^2\}=\sum_{k=1}^{p}m_{kk}=d(1+O(\frac{K^{2+\epsilon}}{\sqrt{n}}))
\end{eqnarray*}
holds with overwhelming probability for any small $\epsilon>0$. Here we have used the fact that
\begin{eqnarray*}
\lambda_{\min}(\hat{D})TrP_J\leq Tr\hat{D}P_J\hat{D}\leq\lambda_{\max}(\hat{D})TrP_J
\end{eqnarray*}
and $TrP_J=d$. By the condition that $d\leq nK^{-3}$, we have
\begin{eqnarray*}
\mathbb{E}\{\sum_{k=1}^pm_{kk}x_{kn}^2\}=d+o(\sqrt{d}K).
\end{eqnarray*}
Let $S_1:=\sum_{k=1}^pm_{kk}(x_{kn}^2-1)$. We have
\begin{eqnarray*}
\mathbb{P}(\sum_{k=1}^pm_{kk}x_{kn}^2\geq d+2\sqrt{d}K)\leq\mathbb{P}(|S_1|\geq \sqrt{d}K)\leq \frac{\mathbb{E}(|S_1|^2)}{dK^2}.
\end{eqnarray*}
And by the assumption on $K$ we also have
\begin{eqnarray*}
\mathbb{E}|S_1|^2=\sum_{k=1}^pm_{kk}^2\mathbb{E}(x_{kn}^2-1)^2=\sum_{k=1}^pm_{kk}^2(\mathbb{E}x_{kn}^4-1)\leq dK.
\end{eqnarray*}
Thus,
\begin{eqnarray*}
\mathbb{P}(|S_1|\geq \sqrt{d}K)\leq \frac{\mathbb{E}(|S_1|^2)}{dK^2}\leq \frac{1}{K}\leq 1/10.
\end{eqnarray*}
 Set $S_2:=|\sum_{k\neq l}m_{kl}x_{kn}x_{ln}|$. Then we have
 \begin{eqnarray*}
 \mathbb{E}S^2_2=2\sum_{k\neq l}m_{kl}^2\leq 2Tr \hat{D}P_J\hat{D}^2P_J\hat{D}\leq ||\hat{D}||_{op}^4TrP_J=2d(1+O(\frac{K^{2+\epsilon}}{\sqrt{n}})).
 \end{eqnarray*}
 By Chebyshev's inequality one has
$$
 \mathbb{P}(S_2\geq 2\sqrt{d}K)\leq 1/10.
$$
 Similarly, we can define $\mathcal{E}_{-}$ as the event $F(\tilde{h}_n)\leq \sqrt{d}-2K$ and use
$$
 \mathbb{P}(\mathcal{E}_{-})\leq \mathbb{P}(S_1\leq d-\sqrt{d}K)+\mathbb{P}(S_2\geq\sqrt{d}K).
$$
 Both terms on the right hand side can be bounded by $1/5$ by the same argument as above. So we conclude the proof.
\end{proof}

\end{document}